\documentclass[11pt,letterpaper]{article}
\usepackage{fullpage,latexsym,verbatim,citesort,times}
\usepackage{amsthm,amsfonts,amsmath,amstext,amssymb,amsopn}
\usepackage{graphicx,epsfig}
\usepackage{xspace}
\usepackage{color}
\usepackage{tikz}
\usepackage{nicefrac}
\usepackage{enumerate}
\usepackage{subfigure}

\usepackage{float}

\usepackage{url}

\newenvironment{gh}{\color{black}}

\newenvironment{ghprime}{\color{black}}

\newenvironment{ghtwo}{\color{black}}

\newtheorem{theorem}{Theorem}[section]

\newtheorem{definition}[theorem]{Definition}

\newcommand {\bit}{\begin{itemize} \item}\newcommand {\eit}{\end{itemize}}
\newcommand {\ben}{\begin{enumerate} \item}\newcommand {\een}{\end{enumerate}}
\newcommand {\bena}{\begin{enumerate}\renewcommand{\labelenumi}{\alph{enumi}.}\item}

\newcommand {\beqn}{\begin{equation}}\newcommand {\eeqn}{\end{equation}}
\newcommand {\beqan}{\begin{eqnarray}}\newcommand {\eeqan}{\end{eqnarray}}
\newcommand {\beqa}{\begin{eqnarray*}}\newcommand {\eeqa}{\end{eqnarray*}}
\newcommand {\barr}{\begin{array}}\newcommand {\earr}{\end{array}}
\newcommand {\bat}{\begin{tabular}}\newcommand {\eat}{\end{tabular}}

\newcommand{\R}{\mathbb{R}}
\newcommand{\Z}{\mathbb{Z}}
\renewcommand{\S}{\mathcal{S}}

\newcommand{\coefx}{\textup{coef}}

\newcommand{\conex}{\textup{cone}}

\newcommand{\coef}[1]{\coefx (#1)}

\newcommand{\cone}[1]{\conex(#1)}

 \normalsize

\newcommand{\U}{\mathcal{U}}

\begin{document}

\title{Optimization over Structured Subsets of\\ Positive Semidefinite Matrices via Column Generation}
\author{
Amir Ali Ahmadi\thanks{Amir Ali Ahmadi and Georgina Hall are partially supported by the Young Investigator Award of the AFOSR and the CAREER Award of the NSF.} \\ Princeton, ORFE \\ \small a\_a\_a@princeton.edu \and  
Sanjeeb Dash\\ IBM Research \\ \small sanjeebd@us.ibm.com \and Georgina Hall\footnotemark[1] \\ Princeton, ORFE \\ \small gh4@princeton.edu}

%\thanks{AAA and GH are partially supported by the Young Investigator Program award of the AFOSR.}

\maketitle

\begin{abstract}
\noindent

We develop algorithms to construct inner approximations of the cone of positive semidefinite matrices via linear programming and second order cone programming. Starting with an initial linear algebraic approximation suggested recently by Ahmadi and Majumdar, we describe an iterative process through which our approximation is improved at every step. This is done using ideas from column generation in large-scale linear programming. We then apply these techniques to approximate the sum of squares cone in a nonconvex polynomial optimization setting, and the copositive cone for a discrete optimization problem.
%The end goal is to replace semidefinite programming with faster and more scalable alternatives.
%In this paper we explore LP and SOCP based approaches to minimizing nonnegative polynomials which are sometimes less expensive than approaches based on solving semidefinite programming subproblems. We also test these approaches on some SDP relaxations of copositive programs. [*TO DO: rewrite the abstract at the end.*] 
\end{abstract}

%\keywords{Column generation, conic optimization, sum of squares programming }

%============================================================
\section{Introduction}
%============================================================

Semidefinite programming is a powerful tool in optimization that is used in many different contexts, perhaps most notably to obtain strong bounds on discrete optimization problems or nonconvex polynomial programs. 
% including polynomial optimization, and obtaining bounds for combinatorial optimization problems, which we discuss in this paper.
One difficulty in applying semidefinite programming is that state-of-the-art general-purpose solvers often cannot solve very large instances reliably and in a reasonable amount of time. As a result, at relatively large scales, one has to resort either to specialized solution techniques and algorithms that employ problem structure, or to easier optimization problems that lead to weaker bounds.
We will focus on the latter approach in this paper.

%for problems where the associated semidefinite program (SDP) is simply intractable or unacceptably slow for current solvers.

At a high level, our goal is to not solve semidefinite programs {\gh (SDPs)} to optimality, but rather replace them with cheaper conic relaxations---\emph{linear and second order cone relaxations} to be precise---that return useful bounds {\gh quickly}. Throughout the paper, we will aim to find lower bounds (for minimization problems); i.e., bounds that certify the distance of a candidate solution to optimality. {\gh Fast, good-quality lower bounds are especially important in the context of branch-and-bound schemes, where one needs to strike a delicate balance between the time spent on bounding and the time spent on branching, in order to keep the overall solution time low.} {\gh Currently, in commercial integer programming solvers, almost all lower bounding approaches using branch-and-bound schemes exclusively produce linear inequalities.} 
%With the current state of affairs in branch-and-bound technology, almost all cutting-plane approaches that produce lower bounds on industrially-sized integer programs are exclusively based on linear programming (LP).
%, or occasionally second order cone programming (SOCP). 
Even though semidefinite cuts are known to be stronger, they are often too expensive to be used even {\gh at the root node of a branch-and-bound tree.} 
%at the root node of branch-and-bound techniques for integer programming. Because of this, 
%many practitioners are wary of using SDP and as a result
%
Because of this, {\gh many} high-performance solvers, e.g., {\gh IBM ILOG CPLEX ~\cite{cplex} and Gurobi \cite{gurobi}}, do not even provide an SDP solver and instead solely work with LP and SOCP relaxations. Our goal in this paper is to offer some tools that exploit the power of SDP-based cuts, while staying entirely in the realm of LP and SOCP. We apply these tools to classical problems in both nonconvex polynomial optimization and discrete optimization. %[*Sanjeeb, confirm what I say about state of BB is OK.*]

%In the field of integer programming for instance, the cutting-plane approaches used on industrial problems are almost exclusively based on linear programming (LP) or second order cone programming (SOCP).

%In the field of sum of squares optimization, however, a sound alternative to sos programming that can avoid SDP and take advantage of the existing mature and high-performance LP/SOCP solvers is lacking. This is precisely what we propose to achieve here.

Techniques that provide lower bounds on minimization problems are precisely those that certify nonnegativity of objective functions on feasible sets. To see this, note that a scalar $\gamma$ is a lower bound on {\gh the minimum value of } a function $f:\mathbb{R}^n\rightarrow\mathbb{R}$ on a set $K\subseteq\mathbb{R}^n,$ if and only if $f(x)-\gamma\geq 0$ for all $x\in K$. As most discrete optimization problems (including {\gh those} in the complexity class NP) can be written as polynomial optimization problems, {\gh the problem of certifying nonnegativity of polynomial functions,} either globally or on basic semialgebraic sets, {\gh is a fundamental one}. A polynomial $p(x)\mathrel{\mathop:}=p(x_1,\ldots,x_n)$ is said to be \emph{nonnegative}, if $p(x)\geq0$ for all $x\in\mathbb{R}^n$. Unfortunately, even in this unconstrained setting, the problem of testing nonnegativity of a polynomial $p$ is NP-hard {\gh even} when its degree equals four. This is an immediate corollary of the fact that checking if a symmetric matrix $M$ is copositive---i.e., if $x^TMx\geq 0, \forall x\geq 0$---is NP-hard.\footnote{Weak NP-hardness of testing matrix copositivity is originally proven by Murty and Kabadi~\cite{nonnegativity_NP_hard}; its strong NP-hardness is apparent from the work of de Klerk and Pasechnik~\cite{dp}.} Indeed, $M$ is copositive if and only if the homogeneous quartic polynomial $p(x)=\sum_{i,j}M_{ij}x_i^2x_j^2$ is nonnegative.

%A powerful sufficient condition for a polynomial to be nonnegative is for it to be a sum of squares of other

%The problem of minimizing a polynomial $p(x)$ where $x = (x_1, \ldots, x_n)$ is a vector of $n$ variables can be represented as the problem of finding the largest number $\lambda$ such that $p(x) - \lambda \geq 0$ for all $x \in \R^n$.
%The problem of verifying if a polynomial $p(x)$ in $n$ variables is nonnegative -- i.e., $p(x)$ satisfies $p(x) \geq 0$ for all $x \in \R^n$ -- is NP-complete for quartic polynomials (this follows from the result of Murty and Kabadi \cite{mk} showing NP-completeness of matrix copositivity verification), and can be hard to solve in practice when the degree is four or more.

Despite this computational complexity barrier, there has been great success in using sum of squares (SOS) programming~\cite{sdprelax},~\cite{lasserre_moment},~\cite{NesterovSquared} to obtain certificates of nonnegativity of polynomials in practical settings. It is known from Artin's solution~\cite{artin1927} to Hilbert's 17th problem that a polynomial $p(x)$ is nonnegative if and only if
\begin{equation}\label{eq:sos} p(x) = \frac{\sum_{i=1}^t q_i^2(x)}{\sum_{i=1}^r g_i^2(x)} \Leftrightarrow (\sum_{i=1}^r g_i^2(x))p(x) = \sum_{i=1}^t q_i^2(x)
\end{equation} for some polynomials $q_1, \ldots, q_t, g_1,\ldots,g_r$.
When $p$ is a quadratic polynomial, then the polynomials $g_i$ are not needed and the polynomials $ q_i$ can be assumed to be linear functions. In this case, by
writing $p(x)$ as
\[ p(x) = \left(\begin{array}{c} 1 \\ x \end{array}\right)^T Q \left(\begin{array}{c} 1 \\ x \end{array}\right), \]
where $Q$ is an $(n+1) \times (n+1)$ symmetric matrix, checking nonnegativity of $p(x)$
reduces to checking the nonnegativity of the eigenvalues of $Q$; i.e., checking if $Q$ is positive semidefinite.

More generally, if the degrees of $q_i$ and $ g_i$ are fixed in (\ref{eq:sos}), then checking for a representation of $p$ of the form in (\ref{eq:sos}) reduces to solving an SDP, whose size depends on the dimension of $x$, and the degrees of $p, q_i$ and $ g_i$ \cite{sdprelax}.
This insight has led to significant progress in certifying nonnegativity of polynomials arising in {\gh many} areas. In practice, the ``first level'' of the SOS hierarchy is often the one used, where the polynomials $g_i$ are left out and one simply checks if $p$ is a sum of squares of other polynomials. In this case already, because of the numerical difficulty of solving large SDPs, the polynomials that can be certified to be nonnegative usually do not have very high degrees or very many variables. For example, finding a sum of squares certificate that a given quartic polynomial over $n$ variables is nonnegative requires solving an SDP involving roughly $O(n^4)$ constraints and a positive semidefinite matrix variable of size $O(n^2) \times O(n^2)$. {\gh Even for a handful of or a dozen variables,} the underlying semidefinite constraints prove to be expensive. Indeed, in the absence of additional structure, most examples in the literature have less than 10 variables.

Recently other systematic approaches to certifying nonnegativity of polynomials have been proposed which lead to less expensive optimization problems than semidefinite programming problems. In particular, Ahmadi and Majumdar~\cite{isos_journal},~\cite{dsos_ciss14} introduce ``DSOS and SDSOS'' optimization techniques, which replace semidefinite programs arising in the nonnegativity certification problem by linear programs and second-order cone programs. Instead of optimizing over the cone of sum of squares polynomials, the authors optimize over two subsets which they call ``diagonally dominant sum of squares'' and ``scaled diagonally dominant sum of squares'' polynomials (see Section~\ref{subsec:dsos.sdsos} for formal definitions). In the language of semidefinite programming, this translates to solving optimization problems over the cone of diagonally dominant matrices and scaled diagonally dominant matrices. These can be done by LP and SOCP respectively. The authors have had notable success with these techniques in different applications. For instance, they are able to run these relaxations for polynomial optimization problems of degree 4 in 70 variables in the order of a few minutes. They have also used their techniques to push the size limits of some SOS problems in controls; examples include stabilizing a model of a humanoid robot with 30 state variables and 14 control inputs~\cite{dsos_cdc14}, or exploring the real-time applications of SOS techniques in problems such as collision-free autonomous motion planning~\cite{OR_isos_OptLetters}.

Motivated by these results, our goal in this paper is to start with DSOS and SDSOS techniques and improve on them. By exploiting {\gh ideas} from column generation in {\ghtwo large-scale linear programming}, and by appropriately interpreting the DSOS and SDSOS constraints, we produce several iterative LP and SOCP-based algorithms that improve the quality of the bounds obtained from the DSOS and SDSOS relaxations. Geometrically, this amounts to optimizing over structured subsets of sum of squares polynomials that are larger than the sets of diagonally dominant or scaled diagonally dominant sum of squares polynomials. For semidefinite programming, this is equivalent to optimizing over structured subsets of the cone of positive semidefinite matrices. An important distinction to make between the DSOS/SDSOS/SOS approaches and our approach, is that our approximations iteratively get larger in the direction of the {\gh given objective function,} unlike the DSOS, SDSOS, and SOS approaches which all try to inner approximate the set of nonnegative polynomials \emph{irrespective} of any particular direction.

{\ghtwo In related literature, Krishnan and Mitchell use linear programming techniques to approximately solve SDPs by taking a semi-infinite LP representation of the SDP and applying column generation \cite{krishnan2006semidefinite}. In addition, Kim and Kojima solve second order cone relaxations of SDPs which are closely related to the dual of an SDSOS program in the case of quadratic programming \cite{kim2003exact}; see Section \ref{sec:cg.overview} for further discussion of these two papers.}
	
The organization of the rest of the paper is as follows.
In the next section, we review relevant notation, and discuss the prior literature on {\gh DSOS and SDSOS programming.}
% two techniques which motivate our column generation ideas.
In Section~\ref{sec:cg.overview}, we give a high-level overview of our column generation approaches in the context of a general SDP.
In Section~\ref{sec:poly.opt}, we describe an application of our ideas to nonconvex polynomial optimization and {\gh present} computational experiments with certain column generation implementations.
{\ghtwo In Section~\ref{sec:stable.set}, we apply our column generation approach to approximate a copositive program arising from a specific discrete optimization application (namely the stable set problem).}
All the work in these sections can be viewed as providing techniques to optimize over subsets of positive semidefinite matrices.
We then conclude in Section~\ref{sec:future} with some future directions, and discuss ideas for column generation which allow one to go beyond subsets of positive semidefinite matrices in the case of polynomial {\ghtwo and copositive } optimization.

%Instead of optimizing a linear function of a matrix variable subject to linear constraints and the restriction that the matrix is positive semidefinite, they replace the positive semidefiniteness condition by either linear constraints (DSOS programming) or second-order cone constraints (SDSOS programming) which are sufficient for positive semidefinitess but are not necessary.
%In other words, they optimize over subsets of the semidefinite cone which are computationally easier to handle. One subset is the set of diagonally dominant matrices, which can be specified via linear constraints on the matrix components, and another is the scaled diagonally dominant matrices, which need second-order cone constraints.
%They thereby solve much simpler optimization problems than SDPs but get weaker lower bounds on the minimum value of a polynomial.

%In this paper, we use the general technique of column generation used to solve very large LPs to extend the DSOS programming and SDSOS programming ideas in Ahmadi and Majumdar \cite{am}.
%In particular, we consider optimization problems defined over different classes of matrices which form strict subsets of all positive semidefinite matrices but contain the class of diagonally dominant matrices.
%We only solve linear programs, but these are harder than the ones solved in \cite{am} but give better bounds for polynomial optimization.
%We also study whether our ideas lead to better bounds compared to DSOS and SDSOS programming in the context of combinatorial optimization problems.

%============================================================
\section{Preliminaries}\label{sec:prelims}
%============================================================
Let us first introduce some notation on matrices. We denote the set of real symmetric $n\times n$ matrices by $S_n$. Given two matrices $A$ and $B$ in $S_n$, we denote their matrix inner product by $A\cdot B := \sum_{i,j}A_{ij}B_{ij} = \mbox{ Trace}(AB)$. The set of symmetric matrices with nonnegative entries is denoted by $N_n$. A symmetric matrix $A$ is \emph{positive semidefinite} (psd) if $x^TAx\geq 0$ for all $x\in\mathbb{R}^n$; this will be denoted by the standard notation $A\succeq 0$, and our notation for the set of $n\times n$ psd matrices is $P_n$. A matrix $A$ is \emph{copositive} if $x^TAx\geq 0$ for all $x\geq 0$. The set of copositive matrices is denoted by $C_n$. All three sets $N_n,P_n,C_n$ are convex cones and we have the obvious inclusion $N_n+P_n\subseteq C_n$. This inclusion is strict if $n\geq 5$~\cite{burer2009difference},~\cite{burer2012copositive}. For a cone $\mathcal{K}$ of matrices in $S_n$, we define its dual cone $\mathcal{K}^*$ as $\{Y \in \S_n: Y\cdot X \geq 0, \ \forall X \in \mathcal{K}\}$.

For a vector variable $x \in \R^n$ and a vector $q \in \Z^n_+$, let a monomial in $x$ be denoted as $x^q\mathrel{\mathop:}= \Pi_{i=1}^n x_i^{q_i}$, and let its degree be $\sum_{i=1}^n q_i$.
A polynomial is said to be \emph{homogeneous} or a \emph{form} if all of its monomials have the same degree. A form $p(x)$ in $n$ variables is nonnegative if $p(x)\geq 0$ for all $x\in\mathbb{R}^n$, or equivalently for all $x$ on the unit sphere in $\mathbb{R}^n$. The set of {\ghtwo \emph{nonnegative}} (or positive semidefinite) forms in $n$ variables and degree $d$ is denoted by $PSD_{n,d}$. A form $p(x)$ is a \emph{sum of squares} (sos) if it can be written as $p(x)=\sum_{i=1}^r q_i^2(x)$ for some forms $q_1,\ldots,q_r$. The set of sos forms in $n$ variables and degree $d$ is {\gh a cone} denoted by $SOS_{n,d}$. We have the obvious inclusion $SOS_{n,d}\subseteq PSD_{n,d}$, which is strict unless $d=2$, or $n=2$, or $(n,d)=(3,4)$~\cite{Hilbert_1888}. Let $z(x,d)$ be the vector of all monomials of degree exactly $d$; it is well known that a form $p$ of degree $2d$ is sos if and only if it can be written as {\gh$p(x)=z^T(x,d)Qz(x,d)$,} for some psd matrix $Q$~\cite{sdprelax},~\cite{PhD:Parrilo}. The size of the matrix $Q$, which is often called the \emph{Gram matrix}, is ${n+d-1 \choose d} \times {n+d-1 \choose d}$. At the price of imposing a semidefinite constraint of this size, one obtains the very useful ability to search and optimize over the convex cone of sos forms via semidefinite programming.

%Let $N_n, S_n, C_n$, stand for the cones of $n\times n$ symmetric nonnegative matrices, diagonally dominant matrices, positive semidefinite matrices (PSD cone) and copositive matrices, respectively.
%
%$D\in D_n$ is a diagonally dominant matrix if $D_{ii} \geq \sum_{j\neq i}|D_{ij}|$ for $i=1, \ldots, n$.
%For a cone of matrices $\mathcal{C} \subseteq \R^{n \times n}$, the dual cone $\mathcal{C}^*$ is defined as $\{Y \in \R^{n\times n}: Y\cdot X \geq 0 \ \forall X \in \mathcal{C}\}$.
%It is well-known that (i) $D_n \subseteq S_n \subseteq C_n$; (ii) $N_n \subseteq C_n$ which also implies that $S_n + N_n \subseteq C_n$; (iii) 
%$S_n^* = S_n$.
%Let $e_1, \ldots e_n$ be the unit vectors in $\R^n$. Let $U_k \subset \R^{n \times n}$ be the set of matrices defined as
%\[ U_k = \{uu^T : u \mbox{ has at most } k \mbox{ nonzero components, each equal to} \pm 1\}. \]
%Clearly $U_k$ is a finite set for each $k=1, \ldots, n$ and it is not hard to see that $D_n = \cone{U_2}$ and therefore $D_n^* = \{X \in \R^{n\times n}: X \cdot V \geq 0 \ \forall V \in U_2\}$.
%Furthermore as $D_n \subseteq S_n$, we have $S_n \subseteq D_n^*$.

\subsection{DSOS and SDSOS optimization}
\label{subsec:dsos.sdsos}

In order to alleviate the problem of scalability posed by the SDPs arising from sum of squares programs, Ahmadi and Majumdar~\cite{isos_journal},~\cite{dsos_ciss14}\footnote{The work in~\cite{isos_journal} is currently in preparation for submission; the one in~\cite{dsos_ciss14} is a shorter conference version of~\cite{isos_journal} which has already appeared. The presentation of the current paper is meant to be self-contained.} recently introduced similar-purpose LP and SOCP-based optimization problems that they refer to as \emph{DSOS and SDSOS programs}. Since we will be building on these concepts, we briefly review their relevant aspects to make our paper self-contained.

%In order to address the problem of scalability posed by SDP, we have recently introduced~\cite{isos_journal},~\cite{dsos_ciss14} alternatives to SOS programming that lead to linear programs (LPs) and second order cone programs (SOCPs). 

The idea in~\cite{isos_journal},~\cite{dsos_ciss14}  is to replace the condition that the Gram matrix $Q$ be positive semidefinite with stronger but {\gh cheaper} conditions in the hope of obtaining more efficient inner approximations to the cone $SOS_{n,d}$. Two such conditions come from the concepts of \emph{diagonally dominant} and \emph{scaled diagonally dominant} matrices in linear algebra. We recall these definitions below.

\begin{definition}\label{def:dd.sdd}
A symmetric matrix {\gh $A=(a_{ij})$} is \emph{diagonally dominant} (dd) if $a_{ii} \geq \sum_{j \neq i} |a_{ij}|$ for all $i$. We say that $A$ is \emph{scaled diagonally dominant} (sdd) if there exists a diagonal matrix $D$, with positive diagonal entries, such that $DAD$ is diagonally dominant.
\end{definition}

We refer to the set of $n \times n$ dd (resp. sdd) matrices as $DD_n$ (resp. $SDD_n$). The following inclusions are a consequence of Gershgorin's circle theorem:
$$DD_n\subseteq SDD_n\subseteq P_n.$$

We now use these matrices to introduce the cones of ``dsos'' and ``sdsos'' forms and some of their generalizations, which all constitute special subsets of the cone of nonnegative forms. We remark that in the interest of brevity, we do not give the original definitions of dsos and sdsos polynomials as they appear in~\cite{isos_journal} (as sos polynomials of a particular structure), but rather an equivalent characterization of them that is more useful for our purposes. The equivalence is proven in~\cite{isos_journal}. 

%We now introduce some cones that are inner approximations of the cone of nonnegative polynomials and that lend themselves to LP and SOCP. In analogy with the representation of sos polynomials in terms of psd matrices (Theorem \ref{thm:sos.sdp}), we define the \emph{dsos} and \emph{sdsos} polynomials in terms of dd and sdd matrices respectively.

 \begin{definition}[\cite{isos_journal,dsos_ciss14}] \label{def:dsos.sdsos.rdsos.rsdsos}
	Recall that $z(x,d)$ denotes the vector of all monomials of degree exactly $d$. A form $p(x)$ of degree $2d$ is said to be
	
	\begin{enumerate}[(i)]
		\item  \emph{diagonally-dominant-sum-of-squares} (dsos) if it admits a representation as\\ $p(x)=z^T(x,d)Qz(x,d)$, where $Q$ is a dd matrix,
		\item  \emph{scaled-diagonally-dominant-sum-of-squares} (sdsos) if it admits a representation as\\ $p(x)=z^T(x,d)Qz(x,d)$, where $Q$ is an sdd matrix,
		\item \emph{$r$-diagonally-dominant-sum-of-squares} ($r$-dsos) if there exists a positive integer $r$ such that \\$p(x) (\sum_{i=1}^n x_i^2)^r$ is dsos,
			\item \emph{$r$-scaled diagonally-dominant-sum-of-squares} ($r$-sdsos) if there exists a positive integer $r$ such that \\$p(x)  (\sum_{i=1}^n x_i^2)^r$ is sdsos.
	\end{enumerate}
\end{definition}
We denote the cone of forms in $n$ variables and degree $d$ that are dsos, sdsos, $r$-dsos, and $r$-sdsos by $DSOS_{n,d}$, $SDSOS_{n,d}$, $rDSOS_{n,d}$, and $rSDSOS_{n,d}$ respectively. The following inclusion relations are straightforward: $$DSOS_{n,d}\subseteq SDSOS_{n,d}\subseteq SOS_{n,d}\subseteq PSD_{n,d},$$
$$rDSOS_{n,d}\subseteq rSDSOS_{n,d}\subseteq PSD_{n,d}, \forall r.$$

The multiplier $(\sum_{i=1}^n x_i^2)^r$ should be thought of as a special denominator in the Artin-type representation in (\ref{eq:sos}). By appealing to some theorems of real algebraic geometry, it is shown in~\cite{isos_journal} that under some conditions, as the power $r$ increases, the sets $rDSOS_{n,d}$ (and hence $rSDSOS_{n,d}$) fill up the entire cone $PSD_{n,d}.$ We will mostly be concerned with the cones $DSOS_{n,d}$ and $SDSOS_{n,d}$, which correspond to the case where $r=0$. From the point of view of optimization, our interest in all of these algebraic notions stems from the following {\gh theorem.}

\begin{theorem}[\cite{isos_journal,dsos_ciss14}]	For any integer $r\geq 0$, the cone $rDSOS_{n,d}$ is polyhedral and the cone $rSDSOS_{n,d}$ has a second order cone representation. Moreover, for any fixed $d$ and $r$, {\gh one can optimize a linear function over $rDSOS_{n,d}$ (resp. $rSDSOS_{n,d}$) by solving a linear program (resp. second order cone program) of size polynomial in $n$.}
% can be done with linear programming (resp. second order cone programming) of size polynomial in $n$.
\end{theorem}

The ``LP part'' of this theorem is not hard to see. The equality $p(x)(\sum_{i=1}^n x_i^2)^r={\gh z^T(x,d)}Qz(x,d)$ gives rise to linear equality constraints between the coefficients of $p$ and the entries of the matrix $Q$ (whose size is polynomial in $n$ for fixed $d$ and $r$). The requirement of diagonal dominance on the matrix $Q$ can also be described by linear inequality constraints on $Q$. The ``SOCP part'' of the statement comes from the fact, shown in~\cite{isos_journal}, that a matrix $A$ is sdd if and only if it can be expressed as 	$$A = \sum_{i< j} M_{2 \times 2}^{ij},$$
where each $ M_{2 \times 2}^{ij}$ is an $n\times n$ symmetric matrix with zeros everywhere except for four entries $M_{ii}, M_{ij}, M_{ji}, M_{jj}$, which must make the $2\times 2$ matrix $\begin{bmatrix} M_{ii} & M_{ij}  \\ M_{ji} & M_{jj} \end{bmatrix}$ symmetric and positive semidefinite. These constraints are \emph{rotated quadratic cone} constraints and can be imposed using SOCP~\cite{socp_alizadeh_goldfarb},~\cite{socp_boyd}:
$$M_{ii}\geq 0, ~\Bigl\lvert\Bigl\lvert\begin{pmatrix}
2M_{ij}\\M_{ii}-M_{jj}
\end{pmatrix}\Bigl\lvert\Bigl\lvert \leq M_{ii}+M_{jj}.$$

We refer to optimization problems with a linear objective posed over the convex cones $DSOS_{n,d}$, $SDSOS_{n,d}$, and $SOS_{n,d}$ as DSOS programs, SDSOS programs, and SOS programs respectively. In general, quality of approximation decreases, while scalability increases, as we go from SOS to SDSOS to DSOS programs. Depending on the size of the application at hand, one may choose one approach over the other. \\
{\ghtwo In related work, Ben-Tal and Nemirovski \cite{ben2001polyhedral} and Vielma, Ahmed and Nemhauser \cite{vielma2010mixed} approximate SOCPs by LPs and produce approximation guarantees.}
%{\ghtwo We remark that in related work, the idea of approximating more complicated convex programs by computationally cheaper ones has been considered by other authors. See, e.g., the work of Ben-Tal and Nemirovski \cite{ben2001polyhedral} and the work of Vielma, Ahmed and Nemhauser \cite{vielma2010mixed}.}

%In this paper, we will be using SOS optimization (Section~\ref{sec:jamming}) and SDSOS optimization (Sections~\ref{sec:barriers} and~\ref{sec:quadrotor}) in our numerical experiments. The reader is referred to~\cite{dsos_cdc14,dsos_ciss14,Ahmadi14} for many numerical examples involving DSOS optimization. We also remark in passing that SDSOS or even DSOS programming enjoy many of the same theoretical (asymptotic) guarantees of SOS programming---results of this nature are proven in~\cite{Ahmadi14}.

%============================================================
\section{Column generation for inner approximation of positive semidefinite cones}\label{sec:cg.overview}
%============================================================

%Column generation~\cite{barnhart1998branch},~\cite{desaulniers2006column} is a technique employed extensively in large-scale, industrial linear and integer programming problems. The idea is to use only a subset of the optimization variables (or dual constraints) and bring in new ones only if they improve the objective function. These algorithms produce a sequence of iterative LPs, often with very smart warm-start and resolve strategies. But what does this have to with nonnegative polynomials?

In this section, we describe a natural approach to apply techniques from the theory of column generation~\cite{barnhart1998branch},~\cite{desaulniers2006column} in large-scale {\ghprime optimization} %programming 
to the problem of optimizing over nonnegative polynomials.  Here is the rough idea: We can think of all SOS/SDSOS/DSOS approaches as ways of proving that a polynomial is nonnegative by writing it as a nonnegative linear combination of certain ``atom'' polynomials that are already known to be nonnegative. For SOS, these atoms are all the squares (there are infinitely many). For DSOS, there is actually a finite number of atoms {\gh corresponding to the extreme rays of the cone of diagonally} dominant matrices (see Theorem~\ref{thm:dsos.corners} below). For SDSOS, once again we have infinitely many atoms, but with a specific structure which is amenable to an SOCP representation. Now the column generation idea is to start with a certain ``cheap'' subset of atoms (columns) and only add new ones---one or a limited number in each iteration---if they improve our desired objective function. This results in a sequence of monotonically improving bounds; we stop the column generation procedure when we are happy with the quality of the bound, or when we have consumed a predetermined budget on time.

In the LP case, after the addition of one or a few new atoms, one can obtain the new optimal solution from the previous solution in much less time than required to solve the new problem from scratch. However, as we show with some examples in this paper, even if one were to resolve the problems from scratch after each iteration (as we do for all of our SOCPs and some of our LPs), the overall procedure is still relatively fast. This is because in each iteration, with the introduction of a constant number $k$ of new atoms, the problem size essentially increases only by $k$ new variables and/or $k$ new constraints. This is in contrast to other types of hierarchies---such as the rDSOS and rSDSOS hierarchies of Definition~\ref{def:dsos.sdsos.rdsos.rsdsos}---that blow up in size by a factor that depends on the dimension in each iteration. 

In the next two subsections we make this general idea more precise. While our focus in this section is on column generation for general SDPs, the next two sections show how the techniques are useful for approximation of SOS programs for polynomial optimization (Section~\ref{sec:poly.opt}), and copositive programs for discrete optimization (Section~\ref{sec:stable.set}). 

\subsection{LP-based column generation}\label{subsec:LP.based.CG}
Consider a general SDP
\begin{equation}\label{eq:sdp}
\begin{aligned} 
\max_{y\in\mathbb{R}^m} &  \quad  b^Ty\\
\text{s.t. } & \quad C-\sum_{i=1}^m y_iA_i\succeq 0,
\end{aligned}
\end{equation} 
with $b\in\mathbb{R}^m, C,A_i\in S_n$ as input, and its dual

\begin{equation}\label{eq:dual.sdp}
\begin{aligned} 
\min_{X\in S_n} &  \quad  C\cdot X\\
\text{s.t. } & \quad A_i\cdot X=b_i, \ i=1,\ldots, m,\\
\  & \quad X\succeq 0.\\
\end{aligned}
\end{equation} 

Our goal is to inner approximate the feasible set of (\ref{eq:sdp}) by increasingly larger polyhedral sets. We consider LPs of the form

\begin{equation}\label{eq:LP.inner}
\begin{aligned} 
\max_{y,\alpha} &  \quad  b^Ty\\
\text{s.t. } & \quad C-\sum_{i=1}^m y_iA_i=\sum_{i=1}^t \alpha_i B_i,\\
\ & \quad \alpha_i\geq 0,\ i=1,\ldots,t.
\end{aligned}
\end{equation} 

Here, the matrices $B_1,\ldots, B_t\in P_n$ are some fixed set of positive semidefinite matrices (our psd ``atoms''). To expand our inner approximation, we will continually add to this list of matrices. This is done by considering the dual LP

\begin{equation}\label{eq:dual.of.LP.inner}
\begin{aligned} 
\min_{X\in S_n} &  \quad  C\cdot X\\
\text{s.t. } & \quad A_i\cdot X=b_i, \ i=1,\ldots, m,\\
\  & \quad X\cdot B_i \geq 0, \ i=1,\ldots, t,
\end{aligned}
\end{equation} 
which in fact gives a polyhedral outer approximation (i.e., relaxation) of the spectrahedral feasible set of the SDP in (\ref{eq:dual.sdp}). If the optimal solution $X^*$ of the LP in (\ref{eq:dual.of.LP.inner}) is already psd, then we are done and have found the optimal value of our SDP. If not, we can use the violation of positive semidefiniteness to extract one (or more) new psd atoms $B_{j}$. Adding such atoms to (\ref{eq:LP.inner}) is called {\em column generation}, and the problem of finding such atoms is called the {\em pricing subproblem}. (On the other hand, if one starts off with an LP of the form (\ref{eq:dual.of.LP.inner}) as an approximation of (\ref{eq:dual.sdp}), then the approach of adding inequalities to the LP iteratively that are violated by the current solution is called a {\em cutting plane} approach, and the associated problem of finding violated constraints is called the {\em separation subproblem}.)
The simplest idea for pricing is to look at the eigenvectors $v_j$ of $X^*$ that correspond to negative eigenvalues. From each of them, one can generate a rank-one psd atom $B_j=v_jv_j^T$, which can be added with a new variable (``column'') $\alpha_j$ to the primal LP in (\ref{eq:LP.inner}), and as a new constraint (``cut'') to the dual LP in (\ref{eq:dual.of.LP.inner}). The subproblem can then be defined as getting the most negative eigenvector, which is equivalent to minimizing the quadratic form $x^TX^*x$ over the unit sphere $\{x|\ ||x||=1\}$. Other possible strategies are discussed later in the paper.

This LP-based column generation idea is rather straightforward, but what does it have to do with DSOS optimization? The connection comes from the extreme-ray description of the cone of diagonally dominant matrices, which allows us to interpret a DSOS program as a particular and effective way of obtaining $n^2$ initial psd atoms.
%To make these ideas more concrete, we start by giving the extreme-ray description of the cone of diagonally dominant matrices. 

Let $\mathcal{U}_{n,k}$ denote the set of vectors in $\mathbb{R}^n$ which have at most $k$ nonzero components, each equal to $\pm 1$, and define $U_{n,k} \subset S_n$ to be the set of matrices 
%\[ U_{n,k} = \{uu^T : u\in\mathbb{R}^n  \mbox{ has at most } k \mbox{ nonzero components, each %equal to} \pm 1\}. \]
\[ U_{n,k} \mathrel{\mathop:}= \{uu^T : u\in\mathcal{U}_{n,k}\}. \]

For a finite set of matrices $T = \{T_1, \ldots, T_t\}$, let $$\cone{T} \mathrel{\mathop:}= \{\sum_{i=1}^t \alpha_i T_i : \alpha_1, \ldots, \alpha_t \geq 0\}.$$ 

%Clearly $U_k$ is a finite set for each $k=1, \ldots, n$ and it is not hard to see that $D_n = \cone{U_2}$ and therefore $D_n^* = \{X \in \R^{n\times n}: X \cdot V \geq 0 \ \forall V \in U_2\}$.

\begin{theorem}[Barker and Carlson~\cite{dd_extreme_rays}]\label{thm:dsos.corners}
$DD_n = \cone{U_{n,2}}.$
\end{theorem}
This theorem tells us that $DD_n$ has exactly $n^2$ extreme rays. It also leads to a convenient representation of the dual cone: $$DD_n^* = \{X \in S_n: v_i^TXv_i \geq 0, \ \mbox{for all vectors $v_i$ with at most 2 nonzero components, each equal to $\pm$1} \}.$$

Throughout the paper, we will be initializing our LPs with the DSOS bound; i.e., our initial set of psd atoms $B_i$ will be the $n^2$ rank-one matrices $u_iu_i^T$ in $U_{n,2}$. This is because this bound is often cheap and effective. Moreover, it guarantees feasibility of our initial LPs (see Theorems~\ref{thm:polyopt.dsos.bound.finite} and \ref{thm:DSOSfinite}), which is {\ghtwo important} for starting column generation. One also readily sees that the DSOS bound can be improved if we were to instead optimize over the cone $U_{n,3}$, which has $n^3$ atoms. However, in settings that we are interested in, we cannot afford to include all these atoms; instead, we will have pricing subproblems that try to pick a useful subset (see Section~\ref{sec:poly.opt}).

We remark that an LP-based column generation idea similar to the one in this section is described in \cite{krishnan2006semidefinite}, where it is used as a subroutine for solving the maxcut problem. The method is comparable to ours inasmuch {\ghtwo as some columns are generated using the eigenvalue pricing subproblem.  However, contrary to us, additional columns specific to max cut are also added to the primal.} The initialization step is also differently done, as the matrices $B_i$ in (\ref{eq:LP.inner}) are initially taken to be in $U_{n,1}$ and not in $U_{n,2}$. (This is equivalent to requiring the matrix $C-\sum_{i=1}^m y_iA_i$ to be diagonal instead of diagonally dominant in (\ref{eq:LP.inner}).)

{\ghprime Another related work is \cite{sherali2002enhancing}. In this paper, the initial LP relaxation is obtained via RLT (Reformulation-Linearization Techniques) as opposed to our diagonally dominant relaxation. The cuts are then generated by taking vectors which violate positive semidefiniteness of the optimal solution as in (\ref{eq:dual.of.LP.inner}). The separation subproblem that is solved though is different than the ones discussed here and relies on an $LU$ decomposition of the solution matrix. }

\subsection{SOCP-based column generation}\label{subsec:socp.based.CG}

In a similar vein, we present an SOCP-based column generation algorithm that in our experience often does much better than the LP-based approach. The idea is once again to optimize over structured subsets of the positive semidefinite cone that are SOCP representable and that are larger than the set $SDD_n$ of scaled diagonally dominant matrices. This will be achieved by working with the following SOCP

\begin{equation}\label{eq:SOCP.inner} {\ghtwo
\begin{aligned} 
\max_{y\in\mathbb{R}^m,a_i^j} &  \quad  b^Ty\\
\text{s.t. } & \quad C-\sum_{i=1}^m y_iA_i=\sum_{i=1}^t V_i \begin{pmatrix} a_i^1 & a_i^2 \\ a_i^2 & a_i^3 \end{pmatrix} V_i^T,\\
\ & \quad \begin{pmatrix} a_i^1 & a_i^2 \\ a_i^2 & a_i^3 \end{pmatrix} \succeq 0,\ i=1,\ldots,t.
\end{aligned}}
\end{equation} 

Here, the positive semidefiniteness constraints on the $2 \times 2$ matrices can be imposed via rotated quadratic cone constraints as explained in Section~\ref{subsec:dsos.sdsos}. {\gh The $n\times 2$ matrices $V_i$ are fixed for all $i=1, \ldots, t$.} Note that this is a direct generalization of the LP in (\ref{eq:LP.inner}), in the case where the atoms $B_i$ are rank-one. To generate a new SOCP atom, we work with the dual of (\ref{eq:SOCP.inner}):

\begin{equation}\label{eq:dual.of.SOCP.inner}
\begin{aligned} 
\min_{X\in S_n} &  \quad  C\cdot X\\
\text{s.t. } & \quad A_i\cdot X=b_i, \ i=1,\ldots, m,\\
\  & \quad V_i^TXV_i \succeq 0, \ i=1,\ldots, t.
\end{aligned}
\end{equation} 

Once again, if the optimal solution $X^*$ is psd, we have solved our SDP exactly; if not, we can use $X^*$ to produce new SOCP-based cuts. For example, by placing the two eigenvectors of $X^*$ corresponding to its two most negative eigenvalues as the columns of an $n\times 2$ matrix $V_{t+1}$, we have produced a new useful atom. (Of course, we can also choose to add more pairs of eigenvectors and add multiple atoms.) As in the LP case, by construction, our bound can only improve in every iteration. 

%{\ghprime In fact, in some cases, the optimal value obtained via the dual SOCP relaxation given in (\ref{eq:dual.of.SOCP.inner}) exactly coincides with the optimal value obtained via the SDP dual given in (\ref{eq:dual.sdp}). See, e.g., \cite{kim2003exact} where it is shown that for a particular class of QCQPs (Quadratically Constrained Quadratic Programs), both the SDP relaxation and the SOCP relaxation are exact.}

We will always be initializing our SOCP iterations with the SDSOS bound. It is not hard to see that this corresponds to the case where we have ${n \choose 2}$ initial $n \times 2$ atoms $V_i$, which have zeros everywhere, except for a 1 in the first column in position $j$ and a 1 in the second column in position {\ghtwo $k> j$}. We denote the set of all such $n\times 2$ matrices by $\mathcal{V}_{n,2}$.

The {\ghtwo first} step of our procedure is carried out already in  \cite{kim2003exact} for approximating solutions to QCQPs. Furthermore, the work in \cite{kim2003exact} shows that for a particular class of QCQPs, its SDP relaxation and its SOCP relaxation (written respectively in the form of (\ref{eq:dual.sdp}) and (\ref{eq:dual.of.SOCP.inner})) are exact.

% $n(n-1)$ initial $n \times 2$ atoms $V_i$, with each column containing exactly one 1, in any position.  (The matrices where both 1s are positioned on the same row are redundant as they can be obtained as a linear combination of other matrices in the set.) 

\begin{figure}[h]
	\begin{center}
		\mbox{
			\subfigure[LP starting with DSOS and adding 5 atoms.]
			{\label{subfig:dsos.iters}\scalebox{0.52}{\includegraphics{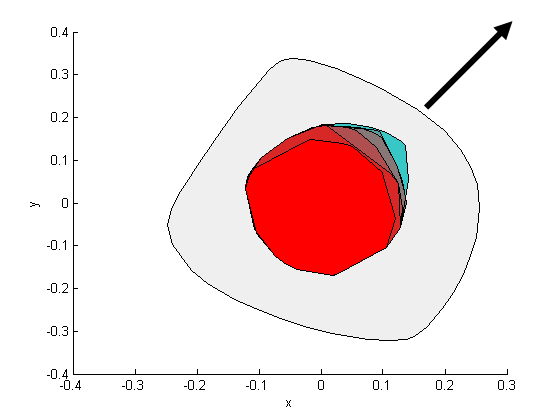}}}}
		\mbox{
			\subfigure[SOCP starting with SDSOS and adding 5 atoms.]
			{\label{subfig:sdsos.iters}\scalebox{0.52}{\includegraphics{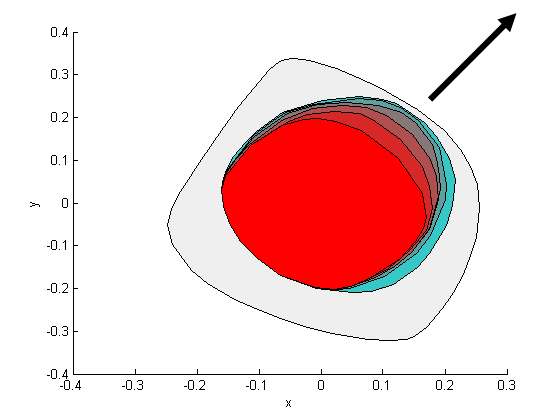}}}
		}
		
		\caption{LP and SOCP-based column generation for inner approximation of a spectrahedron.}
		\label{fig:dsos.sdsos.5atoms}
	\end{center}
	\vspace{-20pt}
\end{figure}

Figure~\ref{fig:dsos.sdsos.5atoms} shows an example of both the LP and SOCP column generation procedures. We produced two $10\times 10$ random symmetric matrices $E$ and $F$. The outer most set is the feasible set of an SDP with the constraint $I+xE+yF\succeq 0.$ (Here, $I$ is the $10\times 10$ identity matrix.) The SDP wishes to maximize $x+y$ over this set. The innermost set in Figure~\ref{subfig:dsos.iters} is the polyhedral set where $I+xE+yF$ is dd. The innermost set in Figure~\ref{subfig:sdsos.iters} is the SOCP-representable set where $I+xE+yF$ is sdd. In both cases, we do 5 iterations of column generation that expand these sets by introducing one new atom at a time. These atoms come from the most negative eigenvector (resp. the two most negative eigenvectors) of the dual optimal solution as explained above. Note that in both cases, we are growing our approximation of the positive semidefinite cone in the direction that we care about (the northeast). This is in contrast to algebraic hierarchies based on ``positive multipliers'' (see the rDSOS and rSDSOS hierarchies in Definition~\ref{def:dsos.sdsos.rdsos.rsdsos} for example), which completely ignore the objective function.

\section{Nonconvex polynomial optimization}\label{sec:poly.opt}

In this section, we apply the ideas described in the previous section to sum of squares algorithms for nonconvex polynomial optimization. In particular, we consider the NP-hard problem of minimizing a form (of degree $\geq 4$) on the sphere. Recall that $z(x,d)$ is the vector of all monomials in $n$ variables with degree $d$.
Let $p(x)$ be a form with $n$ variables and even degree $2d$, and let $\coef{p}$ be the vector of its coefficients with the monomial ordering given by $z(x,2d)$.
% the vector with the same dimension as $z(x,2d)$ defined as follows: if $p(x)$ contains a monomial, then the corresponding component of $\coef{p}$ has a value equal to the coefficient of the monomial.
Thus $p(x)$ can be viewed as $\coef{p}^T z(x,2d)$. Let $s(x) \mathrel{\mathop:}= (\sum_{i=1}^n x_i^2)^d$. With this notation, the problem of minimizing a form $p$ on the unit sphere can be written as
% in this section that we are trying to optimize a form $p(x)$ (with $n$ variables and even degree $2d$) over the unit sphere (minimizing a general polynomial over $\R^n$ is equivalent to this problem) which is equivalent to 

\begin{eqnarray} &\underset{\lambda}{\max} &\lambda \nonumber \\ 
	&\textup{s.t.} & p(x) - \lambda s(x) \geq 0, \forall x\in\mathbb{R}^n. \label{eq:max.lambda}
\end{eqnarray}
%This problem is equivalent to minimizing a general (non-homogeneous) polynomial over $\mathbb{R}^n$.
With the SOS programming approach, the following  SDP is solved to get the largest scalar $\lambda$ and an SOS certificate proving that $p(x) -\lambda s(x)$ is nonnegative: 
\begin{eqnarray} &\underset{\lambda,Y}{\max} & \lambda \nonumber\\
	& \textup{s.t.} &  p(x) - \lambda s(x) = z^T(x,d) Y z(x,d), \label{eq:SDPform}\\
	&&  Y \succeq 0. \nonumber
\end{eqnarray}
The sum of squares certificate is directly read from an eigenvalue decomposition of the solution $Y$ to the SDP above and has the form $$p(x)-\lambda s(x) \geq \sum_{i} (z^T(x,d)u_i)^2,$$ where
$Y = \sum_{i} u_iu_i^T$. 
%If $Y \succeq 0$ does not hold, we do not have a valid SOS certificate.
Since all sos polynomials are nonnegative, the optimal value of the SDP in (\ref{eq:SDPform}) is a lower bound to the optimal value of the optimization problem in (\ref{eq:max.lambda}). Unfortunately, before solving the SDP, we do not have access to the vectors $u_i$ in the decomposition of the optimal matrix $Y$. However, the fact that such vectors exist hints at how we should go about replacing $P_n$ by a polyhedral restriction in (\ref{eq:SDPform}): If the constraint $Y \succeq 0$ is changed to %On the other hand, if we replace the constraint $Y \succeq 0$ by 
\begin{equation}\label{eq:Y=sum.alpha.u.u'}
{\gh Y = \sum_{u \in \U} \alpha_u uu^T, \alpha_u \geq 0},
\end{equation}
where {\gh $\U$ is a finite set}, then (\ref{eq:SDPform}) becomes an LP.
%then for any finite $\U$, we replace $P_n$ by a polyhedral restriction. 
This is one interpretation of Ahmadi and Majumdar's work in \cite{isos_journal,dsos_ciss14} where they replace $P_n$ by $DD_n$. Indeed, this is equivalent to taking $\U=\U_{n,2}$ in (\ref{eq:Y=sum.alpha.u.u'}), as shown in Theorem \ref{thm:dsos.corners}. 
We are interested in extending their results by replacing $P_n$ by larger restrictions than $DD_n$. A natural candidate for example would be obtained by changing $\U_{n,2}$ to $\U_{n,3}$. However, although $\U_{n,3}$ is finite, it contains a very large set of vectors even for small values of $n$ and $d$. For instance, when $n=30$ and $d = 4$, $\U_{n,3}$ has over 66 million elements.
Therefore we use column generation ideas to iteratively expand $\U$ in a manageable fashion. To initialize our procedure, {\ghtwo we would like to start} with good enough atoms to have a feasible LP. The following result guarantees that replacing $Y \succeq 0$ with $Y \in DD_n$ always yields an initial feasible LP in the setting that we are interested in.
\begin{theorem}\label{thm:polyopt.dsos.bound.finite}
For any form $p$ of degree $2d$, there exists $\lambda\in\mathbb{R}$ such that $p(x)-\lambda (\sum_{i=1}^n x_i^2)^d$ is dsos.
\end{theorem}
\begin{proof}
  As before, let $s(x) = (\sum_{i=1}^n x_i^2)^d$.
We observe that the form $s(x)$ is strictly in the interior of $DSOS_{n,2d}$. Indeed, by expanding out the expression we see that we can write $s(x)$ as $z^T(x,d)Qz(x,d)$, where $Q$ is a diagonal matrix with all diagonal entries positive. So $Q$ is in the interior of $DD_{{n+d-1 \choose d}}$, and hence $s(x)$ is in the interior of $DSOS_{n,2d}$. {\ghtwo This implies that for $\alpha>0$ small enough, the form $$(1-\alpha)s(x)+\alpha p(x)$$ will be dsos. Since $DSOS_{n,2d}$ is a cone, the form $$\frac{(1-\alpha)}{\alpha}s(x)+ p(x)$$ will also be dsos. By taking $\lambda$ to be smaller than or equal to} $-\frac{1-\alpha}{\alpha}$, the claim is established. 
\end{proof}

As $DD_n\subseteq SDD_n$, the theorem above implies that replacing $Y \succeq 0$ with $Y \in SDD_n$ also yields an initial feasible SOCP. Motivated in part by this theorem, we will always start our LP-based iterative process with the restriction that $Y \in DD_n$. Let us now explain how we improve on this approximation via column generation.

Suppose we have a set $\U$ of vectors in $\mathbb{R}^n$, whose outerproducts form all of the rank-one psd atoms that we want to consider. This set could be finite but very large, or even infinite. For our purposes $\U$ always includes $\U_{n,2}$, as we initialize our algorithm with the dsos relaxation. Let us consider first the case where $\U$ is finite: $\U = \{u_1, \ldots, u_t\}.$ Then the problem that we are interested in solving is
 %and add elements to $\U$  in each iteration (which in the first iteration corresponds to the set of all vectors such that their outerproduct is in $U_{n,2}$), via column generation. 
%Let us describe this last step more precisely. 
%$\cone{U_{n,3}}$ (and other larger cones) via column generation.
%Thus by solving a series of LPs, we can get an SOS certificate that $p(x) - \lambda s(x)$ for some $\lambda$. We now make this idea more precise.

%Let $z(x,2d)$ have $m$ monomials and consider the problem
\begin{eqnarray*} &\underset{\lambda,\alpha_j}{\max} & \lambda \\
	& \textup{s.t.} &  p(x) - \lambda s(x) = z^T(x,d) Y z(x,d), \\
	&&  Y = \sum_{j=1}^t\alpha_j u_ju_j^T, \  \alpha_j \geq 0 \textup{ for } j=1, \ldots, t.
\end{eqnarray*}
Suppose $z(x,2d)$ has $m$ monomials and let the $i$th monomial in $p(x)$ have coefficient $b_i$, i.e., $\coef{p}= (b_1, \ldots, b_m)^T$.
%
% Let $A_i$ be the matrix such that $A_i \cdot Y = b_i$ implies that $
%z^T(x,d) Y z(x,d)$ has the coefficient $b_i$ for the $i$th monomial in $z(x,2d)$.
Also let $s_i$ be the $i$th entry in $\coef{s(x)}$. We rewrite the previous problem as
%To get the largest value of $\lambda$ such that $p(x) - \lambda$ has an SOS representation, we can solve the problem:
\begin{eqnarray*} &\underset{\lambda,\alpha_j}{\max} & \lambda \\
	& \textup{s.t.} &  A_i \cdot Y + \lambda s_i = b_i \textup{ for } i =1, \ldots, m,\\
	&&  Y = \sum_{j=1}^t \alpha_j u_ju_j^T, \  \alpha_j \geq 0 \textup{ for } j=1, \ldots, t.
\end{eqnarray*}
where $A_i$ is a matrix that collects entries of $Y$ {\ghtwo that contribute to the $i^{th}$ monomial in $z(x,2d)$,} when $z^T(x,d) Y z(x,d)$ is expanded out. 
%Letting $Y_j = u_ju_j^T$, the above is equivalent to
{\gh The above is equivalent to}
\begin{eqnarray} &\underset{\lambda,\alpha_j}{\max} & \lambda  \nonumber \\
	& \textup{s.t.}&  \sum_j \alpha_j(A_i \cdot {\gh u_ju_j^T}) + \lambda s_i = b_i \ \textup{ for } i=1, \ldots, m, \label{eq:polycol}\\ 
	&& \alpha_j \geq 0 \textup{ for } j=1, \ldots, t. \nonumber
\end{eqnarray}
The dual problem is 
\begin{eqnarray*} &\underset{\mu}{\min} & \sum_{i=1}^m \mu_ib_i  \\
	& \textup{s.t.}&  (\sum_{i=1}^m \mu_iA_i) \cdot{\gh u_ju_j^T} \geq 0,\ j=1,\ldots,t, \\
	&&  \sum_{i=1}^m \mu_is_i = 1.
\end{eqnarray*}
In the column generation framework, suppose we consider only a subset of the primal LP variables corresponding to  the matrices ${\gh u_1u_1^T}, \ldots, {\gh u_ku_k^T}$ for some $k < t$ (call this the reduced primal problem).
Let $(\bar\alpha_1, \ldots, \bar\alpha_k)$ stand for an optimal solution of the reduced primal problem and let $\bar \mu = (\bar\mu_1, \ldots, \bar\mu_m)$ stand for an optimal dual solution. If we have 
\begin{equation}\label{eq:sepsub}  (\sum_{i=1}^m \bar{\mu}_iA_i) \cdot {\gh u_ju_j^T} \geq 0 \textup{ for } j=k+1, \ldots, t,\end{equation}
then $\bar\mu$ is an optimal dual solution for the original larger primal problem with columns $1,\ldots,t$. In other words, if we simply set $\alpha_{k+1} = \cdots = \alpha_t = 0$, then the solution of the reduced primal problem becomes a solution of the original primal problem.
On the other hand, if (\ref{eq:sepsub}) is not true, then suppose the condition is violated for some ${\gh u_lu_l^T}$.
We can augment the reduced primal problem by adding the variable $\alpha_l$, and repeat this process.

% under some nondegeneracy assumptions, the solution of this augmented problem will differ from the solution of the previous reduced primal problem and we can repeat this process till (\ref{eq:sepsub}) is true.

Let $B = \sum_{i=1}^m \bar{\mu}_iA_i$.
We can test if (\ref{eq:sepsub}) is false by solving the {\em pricing subproblem}:
\begin{equation}\label{eq:sepsubu}
	\min_{u \in \U} u^TBu.
\end{equation}
If $u^TBu < 0$, then there is an element $u$ in $\U$ such that the matrix $uu^T$ violates the dual constraint written in (\ref{eq:sepsub}). Problem (\ref{eq:sepsubu}) may or may not be easy to solve depending on the set $\U.$ 
For example, an ambitious column generation strategy to improve on dsos (i.e., $\U=\U_{n,2}$), would be to take $\U=\U_{n,n}$; i.e., the set all vectors in $\mathbb{R}^n$ consisting of zeros, ones, and minus ones.
%Therefore, another way of generalizing DSOS programming is by letting $\U$ be a structured set of vectors containing $\U_2$. 
In this case, the pricing problem (\ref{eq:sepsubu})
becomes
\[ \min_{u \in \{0,\pm 1\}^n} u^T B u. \]
Unfortunately, the above problem generalizes the quadratic unconstrained boolean optimization problem (QUBO) and is NP-hard. Nevertheless, there are good heuristics for this problem (see e.g.,~\cite{boros2007local},\cite{dash2013note}) that can be used to find near optimal solutions very fast.
%However, it is very closely related to the Ising Model problem, and it is well-known that though it is hard to obtain a provably optimal
%solution to the above two NP-hard problems, one can often employ heuristics to obtain near optimal solutions very fast~\cite{boros2007local},\cite{dash2013note}.
%Thus one can try to heuristically solve the above pricing problem.
While we did not pursue this pricing subproblem, we did consider optimizing over $\U_{n,3}$. We refer to the vectors in $\U_{n,3}$ as ``triples'' for obvious reasons and generally refer to the process of adding atoms drawn from $U_{n,3}$ as optimizing over ``triples''.

% Indeed, this is a polynomially solvable pricing problem, though still an expensive one.
% Our strategy for managing this subproblem on larger instances is described in the next subsection.

%Our implementation is somewhat straightforward and can be obviously improved, yet we are able to demonstrate that optimizing over triples improves over the best bounds obtained by Ahmadi and Majumdar in a similar amount of time (see Section~\ref{subsec:comput_experim_poly}). 

%Our problem instances are again fully dense and generated in exactly the same way as the $n=10$ example of the previous subsection.

Even though one can theoretically solve (\ref{eq:sepsubu}) with $\U=\U_{n,3}$ in polynomial time by simple enumeration of $n^3$ elements, this is very impractical. %, i.e., the separation problem for triples by simple enumeration, this is impractical because of the large number of triples as noted earlier.
Our simple implementation is a partial enumeration and is implemented as follows. We iterate through the triples (in a fixed order), and test to see whether the condition $u^TBu \geq 0$ is violated by a given triple $u$, and collect such violating triples in a list.
We terminate the iteration when we collect a fixed number of violating triples (say $t_1$). We then sort the violating triples by increasing values of $u^TBu$ (remember, these values are all negative for the violating triples) and select the $t_2$ most violated triples (or fewer if less than $t_2$ are violated overall) and add them to our current set of atoms.
In a subsequent iteration we start off enumerating triples right after the last triple enumerated in the current iteration so that we do not repeatedly scan only the same subset of triples. Although our implementation is somewhat straightforward and can be obviously improved, we are able to demonstrate that optimizing over triples improves over the best bounds obtained by Ahmadi and Majumdar in a similar amount of time (see Section~\ref{subsec:comput_experim_poly}). 

We can also have pricing subproblems where the set $\U$ is infinite. Consider e.g. the case $\U=\mathbb{R}^n$ in (\ref{eq:sepsubu}). In this case, if there is a feasible solution with a negative objective value, then the problem is clearly unbounded below. Hence, we look for a solution with the smallest value of ``violation'' of the dual constraint divided by the norm of the violating matrix. In other words, we want the expression $u^TBu / \textup{norm}(uu^T)$ to be as small as possible, where $\textup{norm}$ is the Euclidean
norm of the vector consisting of all entries of $uu^T$. This is the same as minimizing $u^TBu/||u||^2$. The eigenvector corresponding to the smallest eigenvalue yields such a minimizing solution. This is the motivation behind the strategy described in the previous section for our LP column generation scheme. In this case, we can use a similar strategy for our SOCP column generation scheme. We replace $Y \succeq 0$ by $Y\in SDD_n$ in (\ref{eq:SDPform}) and iteratively expand $SDD_n$ by using the ``two most negative eigenvector technique'' described in Section \ref{subsec:socp.based.CG}.

% We note though that for large matrices $B,$ computation of eigenvectors can be expensive, even though polynomially solvable.

%In such a case, it is standard to restrict the norm of a solution matrix $uu^T$ in the pricing subproblem. It is also common in integer programming to not just work with any solution of the pricing subproblem, but with one
%with the smallest value of ``violation'' of the dual constraint $B \cdot X \geq 0$ divided by the norm of the violating matrix.
%In other words, we want the expression $u^TBu / \textup{norm}(uu^T)$ to be as small as possible, where $\textup{norm}$ is the euclidean
%norm of the vector consisting of all entries of $uu^T$. This is the same as minimizing $u^TBu/||u||^2$; of course, if $\U = \R^n$, then the eigenvector corresponding to the smallest eigenvalue yields such a minimizing solution. This was the strategy briefly described in the previous section. We note though that for large matrices $B,$ computation of eigenvectors can be expensive, even though polynomially solvable.

% we have a polynomial-time pricing subproblem, while still generalizing DSOS programming.

\subsection{Experiments with a 10-variable quartic}\label{subsec:10-var.quartic}

We illustrate the behaviour of these different strategies on an example. Let $p(x)$ be a degree-four form defined on 10 variables, where the components of $\coef{p}$ are drawn independently at random from the normal distribution $\mathcal{N}(0,1)$.
Thus $d=2$ and $n=10$, and the form $p(x)$ is `fully dense' in the sense that $\coef{p}$ has essentially all nonzero components.
In Figure~\ref{fig:errorfig}, we show how the lower bound on the optimal value of $p(x)$ over the unit sphere changes per iteration for different methods. The $x$-axis shows the number of iterations of the column generation algorithm, i.e., the number of times columns are added and the LP (or SOCP) is resolved. The $y$-axis shows the lower bound obtained from each LP or SOCP. Each curve represents one way of adding columns. The three horizontal lines (from top to bottom) represent, respectively, the SDP bound, the 1SDSOS bound and the 1DSOS bound. The curve DSOS$_k$ gives the bound obtained by solving LPs, where the first LP has $Y \in DD_n$ and subsequent columns are generated from a single eigenvector corresponding to the most negative eigenvalue of the dual optimal solution as described in Section~\ref{subsec:LP.based.CG}. The LP triples curve also corresponds to an LP sequence, but this time the columns that are added are taken from $U_{n,3}$ and are more than one in each iteration (see the next subsection). This bound saturates when constraints coming from all elements of $U_{n,3}$ are satisfied. Finally, the curve SDSOS$_k$ gives the bound obtained by SOCP-based column generation as explained just above.
\begin{figure}[H]
\centering
\includegraphics[scale=0.7,trim={3cm 8.5cm 3cm 9cm},clip]{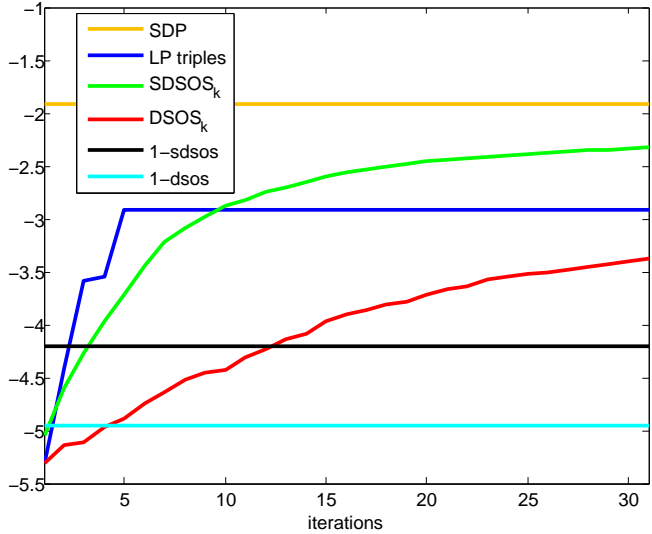}
\caption{Lower bounds for a polynomial of degree 4 in 10 variables obtained via LP and SOCP based column generation}
\label{fig:errorfig}
\end{figure}

\subsection{ Larger computational experiments}\label{subsec:comput_experim_poly}

In this section, we consider larger problem instances ranging from 15 variables to 40 variables: these instances are again fully dense and generated in exactly the same way as the $n=10$ example of the previous subsection. However, contrary to the previous subsection, we only apply our ``triples'' column generation strategy here. This is because the eigenvector-based column generation strategy is too computationally expensive for these problems as we discuss below.
%In this section, we consider larger problem instances ranging from 15 variables to 40 variables and we only apply our ``triples'' column generation strategy. As we discuss below, the eigenvector-based column generation strategy is too computationally expensive for these problems. Our problem instances are again fully dense and generated in exactly the same way as the $n=10$ example of the previous subsection.

To solve the triples pricing subproblem with our partial enumeration strategy, we set $t_1$ to 300,000 and $t_2$ to 5000. Thus in each iteration, we find up to 300,000 violated triples, and add up to 5000 of them. In other words, we augment our LP by up to 5000 columns in each iteration. This is somewhat unusual as in practice at most a few dozen columns are added in each iteration. The logic for this is that primal simplex is very fast in reoptimizing an LP when a small number of additional columns are added to an LP whose optimal basis is known. However, in our context, we observed that the associated LPs are very hard for the simplex routines inside our LP solver (CPLEX 12.4) and take much more time than CPLEX's interior point solver. We therefore use CPLEX's interior point (``barrier'') solver not only for the initial LP but for subsequent LPs after adding columns. Because interior point solvers do not benefit significantly from warm starts, each LP takes a similar amount of time to solve as the initial LP, and therefore it makes sense to add a large number of columns in each iteration to amortize the time for each expensive solve over many columns. 

%In our experiments with stable set (Section~\ref{sec:stable.set}), the LPs are easier to solve ; in that context $t_2$ is smaller.

%In general we observe that even though theoretically the bounds obtained in the first case match the semidefinite programming bounds, in practice one gets worse bounds in the same amount of time.
%This is partially because the matrices $uu^T$ where $u$ is an eigenvector of $B$ are much denser than the elements of $U_{n,3}$.
%For the same reason optimizing over $\U_n$ is very expensive and we do not report computing times with $\U_n$.

Table~\ref{tab:min_hom_opts} is taken from the work of Ahmadi and Majumdar~\cite{isos_journal}, where they report lower bounds on the minimum value of fourth-degree forms on the unit sphere obtained using different methods, and the respective computing times (in seconds).
\begin{table}[H]
\small
\begin{center}
  \begin{tabular}{| c | c | c | c | c | c | c | c | c| c| c|}
    \hline 
 & \multicolumn{2}{|c|}{n=15}  &  \multicolumn{2}{|c|}{n=20} & \multicolumn{2}{|c|}{n=25} &  \multicolumn{2}{|c|}{n=30} & \multicolumn{2}{|c|}{n=40}\\
\cline{2-11}
& bd & t(s) & bd & t(s) & bd & t(s) & bd & t(s) & bd & t(s) \\
\hline
 DSOS & -10.96 & 0.38  & -18.012 & 0.74 & -26.45 & 15.51 &  -36.85 & 7.88 &  -62.30 & 10.68  \\ \hline
 SDSOS  &  -10.43 & 0.53 &  -17.33 & 1.06 &  -25.79 & 8.72 & -36.04 & 5.65 &  -61.25 & 18.66   \\ \hline
1DSOS &  -9.22 & 6.26 &  -15.72 & 37.98 &  -23.58 & 369.08 & NA & NA  & NA & NA  \\ \hline
1SDSOS  & -8.97 & 14.39 & -15.29 & 82.30 & -23.14 & 538.54 & NA & NA & NA & NA   \\ \hline
 SOS & -3.26 & 5.60 & -3.58 & 82.22 & -3.71 & 1068.66 & NA & NA & NA & NA   \\ \hline
  \end{tabular}
    \caption{Comparison of optimal values in \cite{isos_journal} for lower bounding a quartic form on the sphere for varying dimension, along with run times (in seconds). These results are obtained on a 3.4 GHz Windows computer with 16 GB of memory. }
  \label{tab:min_hom_opts}
  \end{center}
\end{table}
In Table~\ref{tab:min_triples}, we give our bounds for the {\ghtwo same problem} instances. We report two bounds, obtained at two different times (if applicable).
In the first case ({\ghtwo rows} labeled R1), the time taken by 1SDSOS in Table~\ref{tab:min_hom_opts} is taken as a limit, and we report the bound from the last column generation iteration occuring {\ghtwo before this time limit}; the 1SDSOS bound is the best non-SDP bound reported in the experiments of Ahmadi and Majumdar. In the {\ghtwo rows} labeled as R2, we take 600 seconds as a limit and report the last bound obtained before this limit.
In a couple of instances ($n=15$ and $n=20$), our column generation algorithm terminates before the 600 second limit, and we report the termination time in this case.
%\begin{table}[H]
%\small
%\begin{center}
%  \begin{tabular}{ | c | c | c | c | c | c | c | c | c | c| c| c| }
%    \hline 
%\  &  \multicolumn{2}{|c|}{n=15}  &  \multicolumn{2}{|c|}{n=20} & \multicolumn{2}{|c|}{n=25} &  \multicolumn{2}{|c|}{n=30} & \multicolumn{2}{|c|}{n=40} \\
%%\cline{2-11}
%\cline{2-11}
%\ & C1 & C2 & C1 & C2 & C1 & C2 & C1 & C2 & C1 & C2 \\
%
%\hline
%time (sec) & 10.96 & 31.19 & 70.70 & 471.39 & 508.63 & 600 &  N/A & 600 & N/A  & 600 \\
%%\hline
%%1-SDSOS bound & -8.97 & \   & -15.29  & \  & -23.14 & \  & N/A & \  & N/A & \   \\
%
%\hline
% 
%triples bound & -6.20 & -5.57  & -12.38  & -9.02 & -20.08 & -20.08 & N/A & -32.38 & N/A & -35.14  \\ \hline
%\end{tabular}
%    \caption{Lower bounds on the optimal value of a form on the sphere for varying degrees of polynomials using Triples on a 2.33 GHz Linux machine with 32 GB of memory.}
%  \label{tab:min_triples}
%  \end{center}
%\end{table}

\begin{table}[H]
	\small
	\begin{center}
		\begin{tabular}{ | c | c | c | c | c | c | c | c | c | c| c| c| }
			\hline 
			\  &  \multicolumn{2}{|c|}{n=15}  &  \multicolumn{2}{|c|}{n=20} & \multicolumn{2}{|c|}{n=25} &  \multicolumn{2}{|c|}{n=30} & \multicolumn{2}{|c|}{n=40} \\
			%\cline{2-11}
			\cline{2-11}
			\ & bd & t(s) & bd & t(s) & bd & t(s) & bd & t(s) & bd & t(s) \\
			
			\hline
			R1 & -6.20 & 10.96 &-12.38 & 70.70 & -20.08 & 508.63 &  N/A & N/A & N/A  & N/A \\
			%\hline
			%1-SDSOS bound & -8.97 & \   & -15.29  & \  & -23.14 & \  & N/A & \  & N/A & \   \\
			
			\hline
			
		R2 & -5.57 & 31.19 &-9.02 & 471.39 & -20.08 & 600 &-32.28 & 600 & -35.14 & 600  \\ \hline
		\end{tabular}
		\caption{Lower bounds on the optimal value of a form on the sphere for varying degrees of polynomials using Triples on a 2.33 GHz Linux machine with 32 GB of memory.}
		\label{tab:min_triples}
	\end{center}
\end{table}

{\ghtwo We observe that in the same amount of time (and even on a slightly slower machine), we are able to consistently beat the 1SDSOS bound, which is the strongest non-SDP bound produced in~\cite{isos_journal}. We also experimented with the eigenvalue pricing subproblem in the LP case, with a time limit of 600 seconds. For $n=25$, we obtain a bound of $-23.46$ after adding only $33$ columns in 600 seconds. For $n=40$, we  are only able to add 6 columns and the lower bound obtained is $-61.49$. Note that this bound is worse than the triples bound given in Table \ref{tab:min_triples}. The main reason for being able to add so few columns in the time limit is that each column is almost fully dense (the LPs for n=25 have 20,475 rows, and 123,410 rows for $n=40$). Thus, the LPs obtained are very hard to solve after a few iterations and become harder with increasing $n$. As a consequence, we did not experiment with the eigenvalue pricing subproblem in the SOCP case as it is likely to be even more computationally intensive.}

\section{Inner approximations of copositive programs and the maximum stable set problem}\label{sec:stable.set}
%============================================================

Semidefinite programming has been used extensively for approximation of NP-hard combinatorial optimization problems. One such example is finding the \emph{stability number} of a graph. A stable set (or independent set) of a graph $G=(V,E)$ is a set of nodes of $G$, no two of which are adjacent. The size of the largest stable set of a graph $G$ is called the stability number (or independent set number) of $G$ and is denoted by $\alpha(G).$  Throughout, {\ghtwo $G$} is taken to be an undirected, unweighted graph on $n$ nodes. It is known that the problem of testing if $\alpha(G)$ is greater than a given integer $k$ is NP-hard \cite{Karp}. Furthermore, the stability number cannot be approximated {\ghtwo to a factor of} $n^{1-\epsilon}$ for any $\epsilon >0$ unless P$=$NP \cite{HaastadJohan}. The natural integer programming formulation of this problem is given by
\begin{equation} \label{IPformulation}
\begin{aligned}
\alpha(G)=&\underset{x_i}{\max} \sum_{i=1}^n x_i\\
&\text{s.t. } x_i+x_j \leq 1, \forall (i,j) \in E,\\
&x_i \in \{0,1\}, \forall i=1,\ldots,n.
\end{aligned}
\end{equation}
Although this optimization problem is intractable, there are several computationally-tractable relaxations that provide upper bounds on the stability number of a graph. For example, the obvious LP relaxation of (\ref{IPformulation}) can be obtained by relaxing the constraint $x_i \in \{0,1\}$ to $x_i \in [0,1]$:
\begin{equation} \label{LPformulation}
\begin{aligned}
LP(G)=&\underset{x_i}{\max} \sum_i x_i\\
&\text{s.t. } x_i+x_j \leq 1, \forall (i,j) \in E,\\
&x_i \in [0,1], \forall i=1,\ldots,n.
\end{aligned}
\end{equation}
{\ghtwo This bound can be improved upon by adding the so-called \emph{clique inequalities} to the LP, which are of the form $x_{i_1}+x_{i_2}+\ldots+x_{i_k} \leq 1$ when nodes $(i_1,i_2,\ldots,i_k)$ form a clique in $G$. Let $C_k$ be the set of all $k$-clique inequalities in $G$. This leads to a hierarchy of LP relaxations:}
\begin{equation}\label{LPkformulation}
\begin{aligned}
LP_k(G)=&\max \sum_i x_i,\\
&x_i \in [0,1], \forall i=1,\ldots,n,\\
& C_2,\ldots,C_k \text{ are satisfied.}
\end{aligned}
\end{equation}
Notice that for $k=2,$ this simply corresponds to (\ref{LPformulation}), in other words, $LP_2(G)=LP(G)$.

{\ghtwo In addition to LPs, there are also semidefinite programming (SDP)} relaxations that provide upper bounds to the stability number. {\ghtwo The most well-known} is perhaps the Lov\'{a}sz theta number $\vartheta(G)$ \cite{Lovasz}, which is defined as the optimal value of the following SDP:
\begin{equation}\label{Lovasz}
\begin{aligned}
\vartheta(G)\mathrel{\mathop{:}}= &\underset{X}{\max } ~J\cdot X\\
&\text{s.t. } I\cdot X=1,\\
&X_{i,j}=0, \forall (i,j) \in E\\
&X \in P_n.
\end{aligned}
\end{equation}
Here $J$ is the all-ones matrix and $I$ is the identity matrix of size $n$.
The Lov\'{a}sz theta number is known to always give at least as good of an upper bound as the LP in (\ref{LPformulation}), even with the addition of clique inequalities of all sizes (there are exponentially many); see, e.g., \cite[Section 6.5.2]{LaurentVall} for a proof. In other words,
$$\vartheta(G) \leq LP_k(G), \forall k.$$

An alternative SDP relaxation for stable set is due to de Klerk and Pasechnik. In \cite{dp}, they show that the stability number can be obtained through a conic linear program over the set of copositive matrices. Namely,
\begin{equation} \label{Copos}
\begin{aligned}
\alpha(G)=&\min_{\lambda} \lambda\\
&\text{s.t. } \lambda(I+A)-J \in C_n,
\end{aligned}
\end{equation}
where $A$ is the adjacency matrix of $G$.
Replacing $C_n$ by the restriction $P_n + N_n$, {\ghtwo one} obtains the aforementioned relaxation through the following SDP
\begin{equation}\label{CoposSDPRelax}
\begin{aligned}
SDP(G)\mathrel{\mathop{:}}= &\min_{\lambda,X} \lambda\\
&\text{s.t. } \lambda(I+A)-J \geq X,\\
&X \in P_n.  
\end{aligned}
\end{equation}
This latter SDP is more expensive to solve than the Lov\'{a}sz SDP (\ref{Lovasz}), but the bound that it obtains is always at least as good (and sometimes strictly better). 
%In fact, it is known that in some cases, $SDP(G)$ strictly improves on $\vartheta(G)$.
A proof of this statement is given in \cite[Lemma 5.2]{dp}, where it is shown that (\ref{CoposSDPRelax}) is an equivalent formulation of an SDP of Schrijver~\cite{sch}, which produces stronger upper bounds than (\ref{Lovasz}).

Another reason for the interest in the copositive approach is that it allows for well-known SDP and LP hierarchies---developed respectively by Parrilo \cite[Section 5]{PhD:Parrilo} and de Klerk and Pasechnik \cite{dp}---that produce a sequence of improving bounds on the stability number. In fact, by appealing to Positivstellensatz results of P\'{o}lya \cite{Polya}, and Powers and Reznick \cite{PR}, de Klerk and Pasechnik show that their LP hierarchy produces the exact stability number in $\alpha^2(G)$ number of steps \cite[Theorem 4.1]{dp}. This immediately implies the same result for stronger hierarchies, such as the SDP hierarchy of Parrilo \cite{PhD:Parrilo}, or the rDSOS and rSDSOS hierarchies of Ahmadi and Majumdar \cite{isos_journal}.

One notable difficulty with the use of copositivity-based SDP relaxations such as (\ref{CoposSDPRelax}) in applications is scalibility. For example, it takes more than 5 hours to solve (\ref{CoposSDPRelax}) when the input is a randomly generated Erd\'{o}s-Renyi graph with 300 nodes and edge probability $p=0.8$. \footnote{The solver in this case is MOSEK~\cite{mosek} and the machine used has 3.4GHz speed and 16GB RAM; see Table \ref{SOCPbounds} for more results. The solution time with the popular SDP solver SeDuMi~\cite{sedumi} e.g. would be several times larger.} Hence, instead of using (\ref{CoposSDPRelax}), we will solve a sequence of LPs/SOCPs generated in an iterative fashion. These easier optimization problems will provide upper bounds on the stability number in a more reasonable amount of time, though they will be weaker than the ones obtained via (\ref{CoposSDPRelax}). 

%We now describe the process through which we iteratively construct the LPs and the SOCPs used to upper bound the stability number.

We will derive both our LP and SOCP sequences from formulation (\ref{Copos}) of the stability number. To obtain the first LP in the sequence, we replace $C_n$ by $DD_n+N_n$ (instead of replacing $C_n$ by $P_n+N_n$ as was done in (\ref{CoposSDPRelax})) and get
\begin{equation}\label{CoposDDRelax}
\begin{aligned} 
DSOS_1(G)\mathrel{\mathop{:}}= &\min_{\lambda,X} \lambda\\
&\text{s.t. } \lambda(I+A)-J \geq X,\\
&X \in DD_n. 
\end{aligned}
\end{equation}
%It is easy to see that this is an LP as optimizing over the set of diagonally dominant matrices can be done using linear programming (see Theorem \ref{blabla}). Furthermore,
This is an LP whose optimal value is a valid {\ghprime upper bound} on the stability number as $DD_n \subseteq P_n$.

\begin{theorem}\label{thm:DSOSfinite}
	The LP in (\ref{CoposDDRelax}) is always feasible. %$DSOS_1(G)$ is always finite.
\end{theorem} 
\begin{proof} 
	We need to show that for any $n\times n$ adjacency matrix $A$, there exists a diagonally dominant matrix $D$, a nonnegative matrix $N$, and a scalar $\lambda$ such that 
	\begin{align} \label{lemma:eq}
	\lambda(I+A)-J=D+N.
	\end{align} 
	 Notice first that $\lambda(I+A)-J$ is a matrix with $\lambda -1$ on the diagonal and at entry $(i,j)$, if $(i,j)$ is an edge in the graph, and with $-1$ at entry $(i,j)$ if $(i,j)$ is not an edge in the graph. If we denote by $d_i$ the degree of node $i$, then let us take $\lambda=n-\min_i d_i+1$ and $D$ a matrix with diagonal entries $\lambda-1$ and off-diagonal entries equal to $0$ if there is an edge, and $-1$ if not. This matrix is diagonally dominant as there are at most $n-\min_i d_i$ minus ones on each row. Furthermore, if we take $N$ to be a matrix with $\lambda-1$ at the entries $(i,j)$ where $(i,j)$ is an edge in the graph, then (\ref{lemma:eq}) is satisfied and $N\geq0$.
\end{proof}
%We will see later that this lemma is crucial for the column generation process to be well defined.\\

Feasibility of this LP is important for us as it allows us to initiate column generation. By contrast, if we were to replace the diagonal dominance constraint by a diagonal constraint for example, the LP could fail to be feasible. This fact has been observed by de Klerk and Pasechnik in \cite{dp} and Bomze and de Klerk in \cite{BdK}.

To generate the next LP in the sequence via column generation, we think of the extreme-ray description of the set of diagonally dominant matrices as explained in Section~\ref{sec:cg.overview}. Theorem~\ref{thm:dsos.corners} tells us that these are given by the matrices in $U_{n,2}$ and so we can rewrite (\ref{CoposDDRelax}) as
\begin{equation}\label{CoposDDRelaxCorners}
\begin{aligned} 
DSOS_1(G)\mathrel{\mathop{:}}= &\min_{\lambda,\alpha_i} \lambda\\
&\text{s.t. } \lambda(I+A)-J \geq X,\\
&X=\sum_{u_iu_i^T\in {U}_{n,2}} \alpha_i u_iu_i^T,\\
&\alpha_i \geq 0,\  i=1,\ldots,n^2.
\end{aligned}
\end{equation}
%Notice that we have $n^2$ variables $\alpha_i$. Even though the cardinality of $\mathcal{U}_{n,2}$ is $2n^2$, these vectors will produce only $n^2$ different matrices $u_iu_i^T$; in other words, $|U_{n,2}|=n^2$.

The column generation procedure aims to add new matrix atoms to the existing set $U_{n,2}$ in such a way that the current bound $DSOS_1$ improves. There are numerous ways of choosing these atoms. We focus first on the cutting plane approach based on eigenvectors. The dual of (\ref{CoposDDRelaxCorners}) is the LP
\begin{equation}\label{CoposDDRelaxCornersDual}
\begin{aligned} 
DSOS_1(G)\mathrel{\mathop{:}}= &\max_{X} J\cdot X,\\
&\text{s.t. } (A+I) \cdot X=1,\\
&X\geq 0,\\
& (u_iu_i^T) \cdot X \geq 0, \forall u_iu_i^T \in {U}_{n,2}.
\end{aligned}
\end{equation}
%The strategy is then the following.
If our optimal solution $X^*$ to (\ref{CoposDDRelaxCornersDual}) is positive semidefinite, then we are obtaining the best bound we can possibly produce, which is the SDP bound of (\ref{CoposSDPRelax}). If this is not the case however, we pick our atom matrix to be the outer product $uu^T$ of the eigenvector $u$ corresponding to the most negative eigenvalue of $X^*$. The optimal value of the LP
\begin{equation}\label{iterativeLP}
\begin{aligned} 
DSOS_2(G)\mathrel{\mathop{:}}= &\max_{X} J \cdot X,\\
&\text{s.t. } (A+I) \cdot X=1,\\
&X\geq 0,\\
& (u_iu_i^T) \cdot X \geq 0, \forall u_iu_i^T \in {U}_{n,2},\\
& (uu^T) \cdot X \geq 0
\end{aligned}
\end{equation}
that we derive is guaranteed to be no worse than $DSOS_1$ as the feasible set of (\ref{iterativeLP}) is smaller than the feasible set of (\ref{CoposDDRelaxCornersDual}). Under mild nondegeneracy assumptions (satisfied, e.g., by uniqueness of the optimal solution to (\ref{CoposDDRelaxCornersDual})), the new bound will be strictly better. By reiterating the same process, we create a sequence of LPs whose optimal values $DSOS_1, DSOS_2, \ldots$ are a nonincreasing sequence of upper bounds on the stability number.

 %portance of Lemma \ref{Lemma:DSOSfinite}. Indeed, if (\ref{CoposDDRelaxCorners}) was infeasible (i.e., the upper bound it produces is infinity), (\ref{CoposDDRelaxCornersDual}) would be infeasible, and we would not be able to generate the matrix atom needed for the construction of the next LP. In particular, this could have been the case if we had required that our variable $X$ be diagonal in (\ref{CoposDDRelax}) rather than diagonally dominant, as observed by de Klerk and Pasechnik in \cite{dp} and Bomze and de Klerk in \cite{BdK}.\\

Generating the sequence of SOCPs is done in an analogous way. Instead of replacing the constraint $X \in P_n$ in (\ref{CoposSDPRelax}) by $X \in DD_n$, we replace it by $X \in SDD_n$ and get
\begin{equation}\label{CoposSDDRelax}
\begin{aligned} 
SDSOS_1(G)\mathrel{\mathop{:}}= &\min_{\lambda,X} \lambda\\
&\text{s.t. } \lambda(I+A)-J \geq X,\\
&X \in SDD_n.
\end{aligned}
\end{equation}
Once again, we need to reformulate the problem in such a way that the set of scaled diagonally dominant matrices is described as some combination of psd ``atom" matrices. In this case, we can write any matrix $X \in SDD_n$ as $$X=\sum_{V_i\in\mathcal{V}_{n,2}}V_i\begin{pmatrix} a_i^1 & a_i^2 \\ a_i^2 & a_i^3\end{pmatrix}V_i^T,$$ where $a_i^1,a_i^2, a_i^3$ are variables making the $2\times 2$ matrix psd, and the $V_i$'s are our atoms. Recall from Section~\ref{sec:cg.overview} that the set $\mathcal{V}_{n,2}$ consists of all $n\times 2$ matrices which have zeros everywhere, except for a 1 in the first column in position $j$ and a 1 in the second column in position $k\neq j$. This gives rise to an equivalent formulation of (\ref{CoposSDDRelax}):
\begin{equation}\label{CoposSDDRelaxCorners}
\begin{aligned} 
SDSOS_1(G)\mathrel{\mathop{:}}= &\min_{\lambda,a_i^j} \lambda\\
&\text{s.t. } \lambda(I+A)-J \geq X\\
&X=\sum_{V_i\in\mathcal{V}_{n,2}}V_i\begin{pmatrix} a_i^1 & a_i^2 \\ a_i^2 & a_i^3\end{pmatrix}V_i^T\\
&\begin{pmatrix} a_i^1 & a_i^2 \\ a_i^2 & a_i^3\end{pmatrix} \succeq 0, \  i=1,\ldots,\binom{n}{2}.
\end{aligned}
\end{equation}
Just like the LP case, we now want to generate one (or more) $n \times 2$ matrix $V$ to add to the set $\{V_i\}_i$ so that the bound $SDSOS_1$ improves. We do this again by using a cutting plane approach originating from the dual of (\ref{CoposSDDRelaxCorners}):
\begin{equation}\label{CoposSDDRelaxCornersDual}
\begin{aligned} 
SDSOS_1(G)\mathrel{\mathop{:}}= &\max_{X} J \cdot X\\
&\text{s.t. } (A+I) \cdot X=1,\\
&X\geq 0,\\
& V_i^T\cdot X V_i \succeq 0, \ i=1,\ldots,\binom{n}{2}.
\end{aligned}
\end{equation} 
%As $DSOS_1(G)$ is always finite and $SDSOS_1(G)\leq DSOS_1(G)$, it follows that $SDSOS_1(G)$ is always finite as well. This entails that an optimal solution $X^*$ to (\ref{CoposSDDRelaxCornersDual}) exists.

Note that strong duality holds between this primal-dual pair as it is easy to check that both problems are strictly feasible. We then take our new atom to be $$V=(w_1 ~w_2),$$ where $w_1$ and $w_2$ are two eigenvectors corresponding to the two most negative eigenvalues of $X^*$, the optimal solution of (\ref{CoposSDDRelaxCornersDual}). If $X^*$ only has one negative eigenvalue, we add a linear constraint to our problem; if $X^* \succeq 0$, then the bound obtained is identical to the one obtained through SDP (\ref{CoposSDPRelax}) and we cannot hope to improve. Our next iterate is therefore
\begin{equation}\label{iterativeSOCP}
\begin{aligned} 
SDSOS_2(G)\mathrel{\mathop{:}}= &\max_{X} J \cdot X\\
&\text{s.t. } (A+I) \cdot X=1,\\
&X\geq 0,\\
& V_i^T\cdot X V_i \succeq 0, \ i=1,\ldots,\binom{n}{2},\\
& V^T\cdot X V \succeq 0.
\end{aligned}
\end{equation} 
Note that the optimization problems generated iteratively in this fashion always remain SOCPs and their optimal values form a nonincreasing sequence of upper bounds on the stability number.

To illustrate the column generation method for both LPs and SOCPs, we consider the complement of the Petersen graph as shown in Figure \ref{fig:Petersen} as an example. The stability number of this graph is 2 and one of its maximum stable sets is designated by the two white nodes. In Figure \ref{Petersenbounds}, we compare the upper bound obtained via (\ref{CoposSDPRelax}) and the bounds obtained using the iterative LPs and SOCPs as described in (\ref{iterativeLP}) and (\ref{iterativeSOCP}). 

\begin{figure}[h]
	\begin{center}
		\mbox{
			\subfigure[The complement of Petersen Graph]
			{\label{fig:Petersen}\scalebox{0.3}{\includegraphics[trim={1cm 0cm 2cm 1cm},clip]{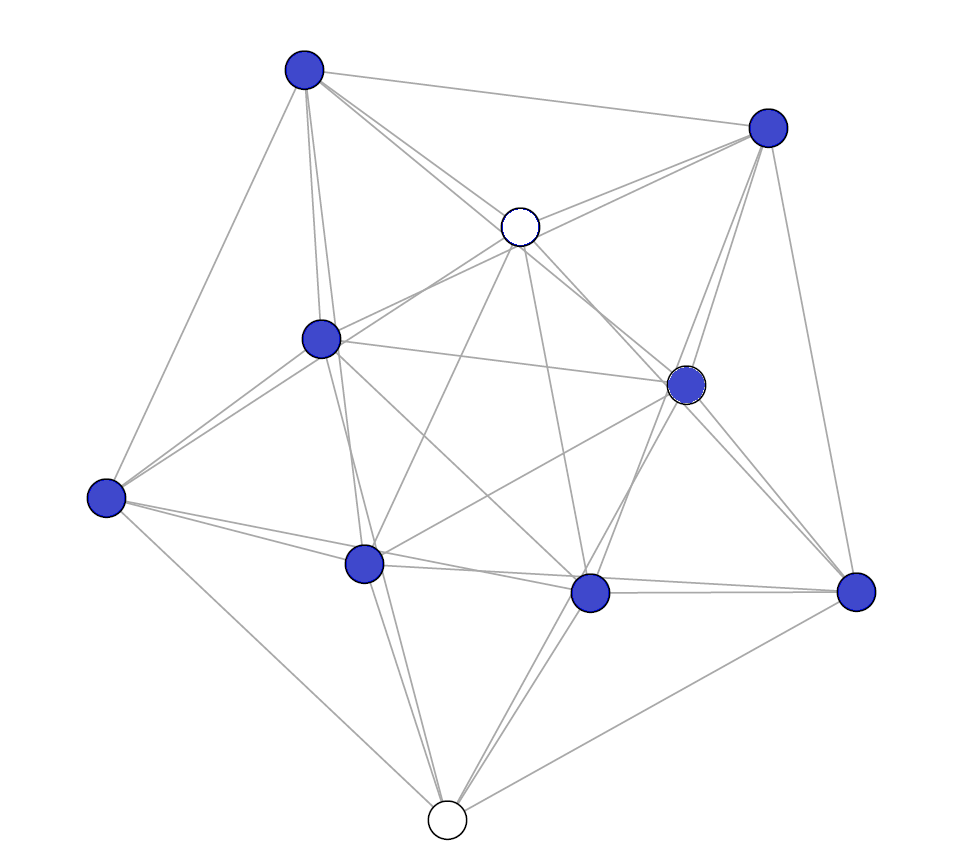}}}}
		\mbox{
			\subfigure[Upper bounds on the stable set number $\alpha(G)$]
			{\label{Petersenbounds}\scalebox{0.63}{\includegraphics[trim={3.5cm 8.5cm 3.5cm 9cm},clip]{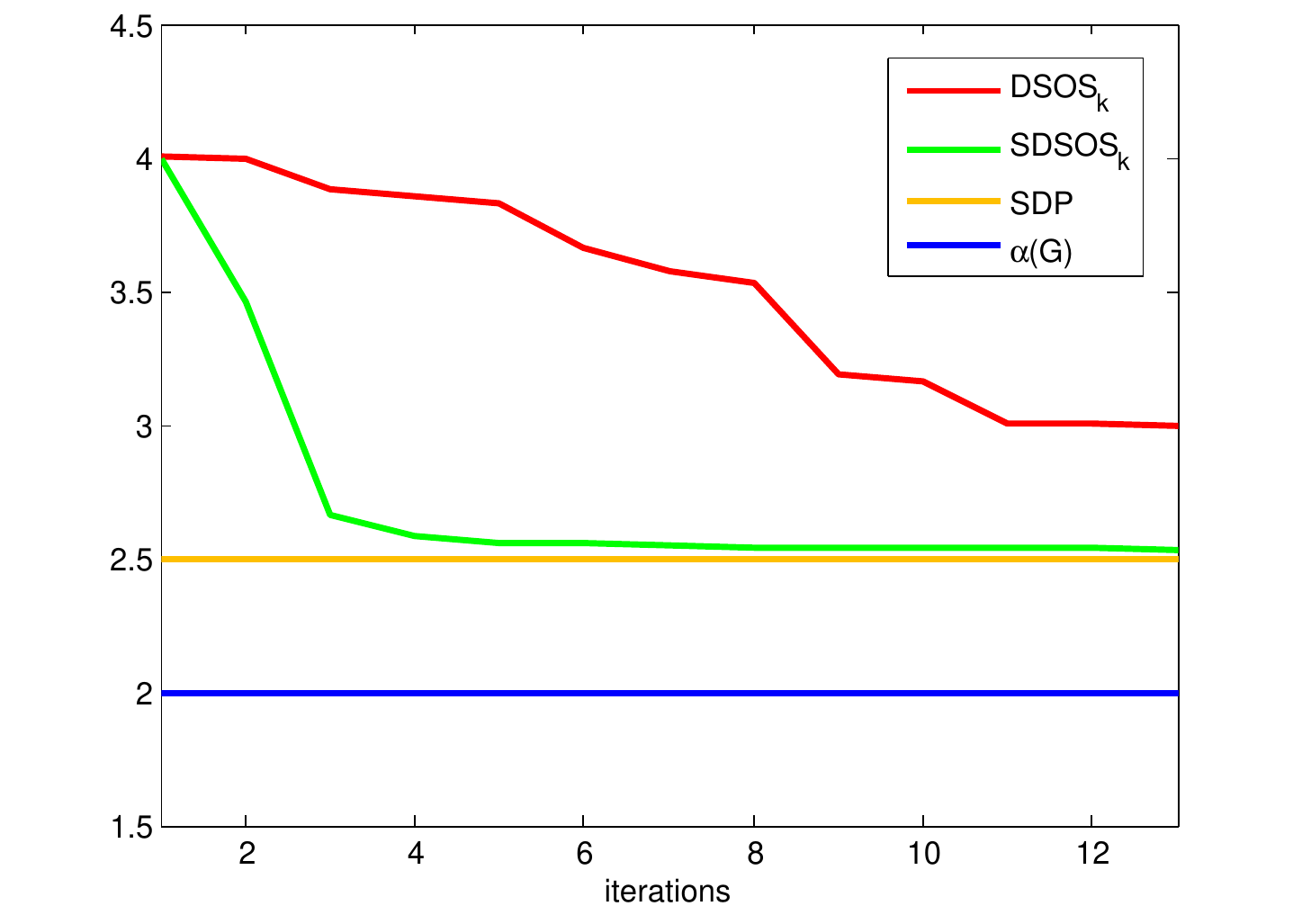}}}
		}
		
		\caption{Bounds obtained through SDP (\ref{CoposSDPRelax}) and iterative SOCPs and LPs for the complement of the Petersen graph.}
		\label{fig:test}
	\end{center}
	\vspace{-20pt}
\end{figure}

%
%\begin{figure}[H]
%	\centering
%	\begin{subfigure}{.5\textwidth}
%		\centering
%		\includegraphics[scale=0.55,trim={3cm 2.5cm 2cm 2.5cm},clip]{prettygraph}
%		\caption{The complement of Petersen Graph}
%		\label{fig:Petersen}
%	\end{subfigure}%
%	\begin{subfigure}{.5\textwidth}
%		\centering
%		\includegraphics[scale=0.65,trim={3.5cm 8.5cm 3.5cm 9cm},clip]{PetersenBounds.pdf}
%		\caption{Upper bounds on the stable set number $\alpha(G)$}
%		\label{Petersenbounds}
%	\end{subfigure}
%	\caption{Bounds obtained through SDP (\ref{CoposSDPRelax}) and iterative SOCPs and LPs for the complement of the Petersen graph}
%	\label{fig:test}
%\end{figure}

Note that it takes 3 iterations for the SOCP sequence to produce an {\ghprime upper bound} strictly within one unit of the actual stable set number (which would immediately tell us the value of $\alpha$), whereas it takes 13 iterations for the LP sequence to do the same. It is also interesting to compare the sequence of LPs/SOCPs obtained through column generation to the sequence that one could obtain using the concept of $r$-dsos/$r$-sdsos polynomials. Indeed, LP (\ref{CoposDDRelax}) (resp. SOCP (\ref{CoposSDDRelax})) can be written in polynomial form as 
\begin{equation}
\begin{aligned}
DSOS_1(G) \text{ (resp. $SDSOS_1(G)$)}=&\min_{\lambda} \lambda\\
&\text{s.t. } \begin{pmatrix} x_1^2 \\ \vdots \\ x_n^2 \end{pmatrix}^T (\lambda (I+A) -J) \begin{pmatrix} x_1^2 \\ \vdots \\ x_n^2 \end{pmatrix} \text{ is dsos (resp. sdsos).}
%\begin{pmatrix}x_1^2  end{pmatrix} (\lambda(I+A)-J) \begin{pmatrix}x_1^2 \\ \vdots \\ x_n^2 \end{pmatrix} \text{ is dsos (resp. sdsos)}
\end{aligned}
\end{equation}
Iteration $k$ in the sequence of LPs/SOCPs would then correspond to requiring that this polynomial be $k$-dsos or $k$-sdsos. For this particular example, we give the 1-dsos, 2-dsos, 1-sdsos and 2-sdsos bounds in Table \ref{Table1dsos1sdsos}. 

\begin{table}[h]
	\centering
	\begin{tabular}{|c|c|c|}
		\hline
		Iteration & $r$-dsos bounds & $r$-sdsos bounds\\
		\hline
		$r=0$ & 4.00 & 4.00 \\
		\hline
		$r=1$ & 2.71 & 2.52\\
		\hline 
		$r=2$ & 2.50 & 2.50 \\
		\hline
	\end{tabular}
	\caption{Bounds obtained through rDSOS and rSDSOS hierarchies.}
	\label{Table1dsos1sdsos}
	
\end{table}

Though this sequence of LPs/SOCPs gives strong upper bounds, each iteration is more expensive than the iterations done in the column generation approach. Indeed, in each of the column generation iterations, only one constraint is added to our problem, whereas in the rDSOS/rSDSOS hierarchies, the number of constraints is roughly multiplied by $n^2$ at each iteration.

%
%\begin{figure}[H]\label{Petersen}
%\centering
%\includegraphics[scale=0.55,trim={2.7cm 2.7cm 2.7cm 2.7cm},clip]{prettygraph}
%\caption{Complement of Petersen Graph}
%\end{figure}
%%\begin{table}\label{Petersenbounds}
%\centering
%\begin{tabular}{c|c||c|c}
%            \hline \hline
%	\multicolumn{4}{c}{SDP bound}\\
%\hline
%	\multicolumn{4}{c}{2.5}\\
%\hline \hline
%	\multicolumn{4}{c}{LP bounds} \\
%            \cline{1-4}
%             Iteration & Bound & Iteration & Bound  \\
%            \hline
%1 & 4 & 8 & 3.529\\
% 2 & 3.997 &  9 & 3.193\\\
%3 & 3.885 & 10 & 3.166\\
%4 & 3.854 & 11 & 3.006\\
%5 & 3.825 & 12 & 3.003\\
%6  &  3.665 & 13 & 2.99\\
%7 & 3.573 &  &\\
%\hline
%        \end{tabular}
%\end{table}

Finally, we investigate how these techniques perform on graphs with a large number of nodes, where the SDP bound cannot be found in a reasonable amount of time. The graphs we test these techniques on are Erd\"os-R\'{e}nyi graphs $ER(n,p)$; i.e. graphs on $n$ nodes where an edge is added between each pair of nodes independently and with probability $p$. In our case, we take $n$ to be between $150$ and $300$, and $p$ to be either $0.3$ or $0.8$ so as to experiment with both medium and high density graphs.\footnote{All instances used for these tests are available online at \url{http://aaa.princeton.edu/software}.} 

%In very sparse graphs, LP based methods usually perform much better than SDP based methods.

In Table \ref{SOCPbounds}, we present the results of the iterative SOCP procedure and contrast them with the SDP bounds. The third column of the table contains the SOCP upper bound obtained through (\ref{CoposSDDRelaxCornersDual}); the solver time needed to obtain this bound is given in the fourth column. The fifth and sixth columns correspond respectively to the SOCP iterative bounds obtained after 5 mins solving time and 10 mins solving time. Finally, the last two columns chart the SDP bound obtained from (\ref{CoposSDPRelax}) and the time in seconds needed to solve the SDP. All SOCP and SDP experiments were done using Matlab, the solver MOSEK \cite{mosek}, the SPOTLESS toolbox \cite{SPOT_Megretski}, and a computer with 3.4 GHz speed and 16 GB RAM.

\begin{table}[H] 
	\centering
	\begin{tabular}{c|c|cc|c|c|cc}
		\hline
		n & p & $SDSOS_1$ & time (s) & $SDSOS_k$ (5 mins) & $SDSOS_k$ (10 mins) &$SDP(G)$ & time (s)\\
		\hline
		150 & 0.3 & 105.70 & 1.05 & 39.93 & 37.00 & 20.43 & 221.13\\
		150 & 0.8 & 31.78 & 1.07 & 9.96& 9.43 & 6.02 & 206.28\\
		200 & 0.3 & 140.47 & 1.84& 70.15& 56.37 & 23.73 & 1,086.42\\
		200 & 0.8 & 40.92 & 2.07 & 12.29 & 11.60 & 6.45 & 896.84\\
		250 & 0.3 & 176.25 & 3.51 & 111.63 & 92.93 & 26.78 & 4,284.01 \\
		250 & 0.8 & 51.87 & 3.90 & 17.25 & 15.39 & 7.18 & 3,503.79\\
		300 & 0.3 & 210.32 & 5.69 & 151.71 & 134.14 & 29.13 & 32,300.60\\
		300 & 0.8 & 60.97 & 5.73 & 19.53 & 17.24 & 7.65 & 20,586.02\\
		\hline
	\end{tabular}
	\caption{SDP bounds and iterative SOCP bounds obtained on ER(n,p) graphs.}
	\label{SOCPbounds}
\end{table} 

From the table, we note that it is better to run the SDP rather than the SOCPs for small $n$, as the bounds obtained are better and the times taken to do so are comparable. However, as $n$ gets bigger, the SOCPs become valuable as they provide good upper bounds in reasonable amounts of time. For example, for $n=300$ and $p=0.8$, the SOCP obtains a bound that is only twice as big as the SDP bound, but it does so 30 times faster. {\gh The sparser graphs don't do as well, a trend that we will also observe in Table \ref{LPbounds}.}
%The sparser graphs don't do as well, but the bounds obtained in $1/30^{th}$ of the time are still only within 4 times the ones obtained with the SDP. 
Finally, notice that the improvement in the first 5 mins is significantly better than the improvement in the last 5 mins. This is partly due to the fact that the SOCPs generated at the beginning are sparser, and hence faster to solve. 

In Table \ref{LPbounds}, we present the results of the iterative LP procedure used on the same instances. All LP results were obtained using a computer with 2.3 GHz speed and 32GB RAM and the solver CPLEX 12.4 \cite{cplex}. The third and fourth columns in the table contain the LP bound obtained with (\ref{CoposDDRelaxCornersDual}) and the solver time taken to do so. Columns 5 and 6 correspond to the LP iterative bounds obtained after 5 mins solving time and 10 mins solving time using the eigenvector-based column generation technique (see discussion around (\ref{iterativeLP})). The seventh and eighth columns are the standard LP bounds obtained using (\ref{LPkformulation}) and the time taken to obtain the bound. Finally, the last column gives bounds obtained by column generation using  ``triples'', as described in Section \ref{subsec:comput_experim_poly}. In this case, we take $t_1=300,000$ and $t_2=500$.

\begin{table}[H] 
	\centering
	\begin{tabular}{c|c|cc|c|c|cc|c}
		\hline
		n & p & $ DSOS_1 $ & time (s) & $DSOS_k$ (5m) & $DSOS_k$ (10m) & $LP_2 $& time (s) &  $LP_{triples}$ (10m)\\
		\hline
		150 & 0.3 & 117 & $<1$ & 110.64 &  110.26 & 75 & $<1$ & 89.00 \\
		150 & 0.8 & 46 & $<1$ & 24.65 & 19.13 & 75 & $<1$ &23.64  \\
		200 & 0.3 & 157 & $<1 $& 147.12 &  146.71 & 100 & $<1$ & 129.82\\
		200 & 0.8 & 54 & $<1$ & 39.27 & 36.01 & 100 & $<1$ & 30.43 \\
		250 & 0.3 & 194 & $<1$ & 184.89 & 184.31 & 125 & $<1$ & 168.00\\
		250 & 0.8 &  68 & $<1$ & 55.01 & 53.18 & 125 & $<1$ & 40.19\\
		300 & 0.3 & 230 & $<1$ & 222.43 & 221.56 & 150 & $<1$&  205.00 \\
		300 & 0.8 & 78 & $<1$ & 65.77 & 64.84 &  150 & $<1$& 60.00\\
		\hline
	\end{tabular}
	\caption{LP bounds obtained on the same $ER(n,p)$ graphs.}
	\label{LPbounds}
\end{table} 

We note that in this case the upper bound with triples via column generation does better for this range of $n$ than eigenvector-based column generation in the same amount of time. Furthermore, the iterative LP scheme seems to perform better in the dense regime. In particular, the first iteration does significantly better than the standard LP for $p=0.8$, even though both LPs are of similar size. This would remain true even if the 3-clique inequalities were added as in (\ref{LPkformulation}), since the optimal value of $LP_3$ is always at least $n/3$. This is because the vector $(\frac{1}{3},\ldots,\frac{1}{3})$ is feasible to the LP in (\ref{LPkformulation}) with $k=3$. Note that this LP would have order $n^3$ constraints, which is more expensive than our LP. On the contrary, for sparse regimes, the standard LP, {\ghtwo which hardly takes any time to solve,} gives better bounds than ours.

Overall, the high-level conclusion is that running the SDP is worthwhile for small sizes of the graph. As the number of nodes increases, column generation becomes valuable, providing upper bounds in a reasonable amount of time. Contrasting Tables \ref{SOCPbounds} and \ref{LPbounds}, our initial experiments seem to show that the iterative SOCP bounds are better than the ones obtained using the iterative LPs. It may be valuable to experiment with different approaches to column generation however, as the technique used to generate the new atoms seems to impact the bounds obtained.

\section{Conclusions and future research}\label{sec:future}

For many problems of discrete and polynomial optimization, there are hierarchies of SDP-based sum of squares algorithms that produce provably optimal bounds in the limit~\cite{sdprelax},~\cite{lasserre_moment}. However, these hierarchies can often be expensive computationally. In this paper, we were interested in problem sizes where even the first level of the hierarchy is too expensive, and hence we resorted to algorithms that replace the underlying SDPs with LPs or SOCPs. We built on the recent work of Ahmadi and Majumdar on DSOS and SDSOS optimization~\cite{isos_journal},~\cite{dsos_ciss14}, which serves exactly this purpose. {\ghtwo We showed that by using ideas from linear programming column generation,} the performance of their algorithms is improvable. We did this by iteratively optimizing over increasingly larger structured subsets of the cone of positive semidefinite matrices, without resorting to the more expensive rDSOS and rSDSOS hierarchies.

%In this paper, based on the inner representation of the cone of diagonally dominant matrices, we proposed an extension of the DSOS and SDOS-programming ideas developed in 
%By considering specific, structured subsets of the semidefinite cone, and using column generation to optimize over these structured cones (or using cutting plane approaches to optimize over analogous structured relaxations of the semidefinite cone), 

%we demonstrated that one could get better bounds than with DSOS or SDSO programming without resorting to the next levels of the associated heirarchies or solving semidefinite programs, and by spending a reasonable amount of extra time.

There is certainly a lot of room to improve our column generation algorithms. In particular, we only experimented with a few types of pricing subproblems and particular strategies for solving them. {\ghtwo The success of column generation often comes} from good ``engineering'', which fine-tunes the algorithms to the problem at hand. Developing warm-start strategies for our iterative SOCPs for example, would be a very useful problem to work on in the future.

% we feel that our column generation ideas need to be applied differently to different problems, based on problem characteristics. Much more computational experiments are needed to master a good ``engineering'' understanding of the choice of the pricing subproblems and efficient strategies for solving them. This 

%Our separation algorithms are not best possible, and we believe it should be possible to speed up our column generation algorithms with a more careful implementation.

Here is another interesting research direction, which for illustrative purposes we outline for the problem studied in Section~\ref{sec:poly.opt}; i.e., minimizing a form on the sphere. Recall that given a form $p$ of degree $2d$, we are trying to find the largest $\lambda$ such that $p(x)-\lambda (\sum_{i=1}^n x_i^2)^d$ is a sum of squares. Instead of solving this sum of squares program, we looked for the largest $\lambda$ for which we could write $p(x)-\lambda$ as a conic combination of a certain set of nonnegative polynomials. These polynomials for us were always either a single square or a sum of squares of polynomials.
%Each column in our column generation framework represents a very simple nonnegative polynomial, namely one that is a square of a linear combination of the monomials of degree $d$,
%In the context of polynomial optimization, we do not need to restrict ourself to obtaining a sum of squares representation via column generation. Each column in our column generation framework represents a very simple nonnegative polynomial, namely one that is a square of a linear combination of the monomials of degree $d$, i.e., it has the form $(z(x,d) \cdot u)^2$.
%In other words, we are trying to find a representation of $p(x)-\lambda s(x)$ as $\sum_{i} \alpha_i p_i(x)^2$ where $p_i(x)$ is a polynomial of degree $d/2$ as we know that $p_i(x)^2$ is nonnegative.
%
There are polynomials, however, that are nonnegative but not representable as a sum of squares. Two classic examples~\cite{MotzkinSOS_temp},~\cite{Choi_Lam_extremalPSDforms} are the Motzkin polynomial
\[ M(x,y,z) = x^6 + y^4z^2 + y^2z^4 - 3x^2y^2z^2, \]
and the Choi-Lam polynomial
\[ CL(w,x,y,z) = w^4 + x^2y^2 + y^2z^2 + x^2z^2 - 4wxyz. \]

%%There are polynomials, however,  such as the Motzkin polynomial \cite{MotzkinSOS_temp}
%\[ M(x,y,z) = x^6 + y^4z^2 + y^2z^4 - 3x^2y^2z^2 \] that are nonnegative but not representable as a sum of squares of other polynomials, i.e., not {\em SOS-representable}.
%Similarly, the Choi-Lam polynomial \cite{Choi_Lam_extremalPSDforms}
%\[ CL(w,x,y,z) = w^4 + x^2y^2 + y^2z^2 + x^2z^2 - 4wxyz \]
%is nonnegative for all $(w,x,y,z) \in \R^4$ but is not SOS-representable.

Either of these polynomials can be shown to be nonnegative using the arithmetic mean-geometric mean (am-gm) inequality, which states
%inequality of arithmetic means and geometric means (we call it the am-gm inequality) which states
that if $x_1, \ldots, x_k \in \R$, then 
\[ x_1, \ldots, x_k \geq 0 \Rightarrow (\sum_{i=1}^k x_i)/k \geq (\Pi_{i=1}^k x_i)^{\frac{1}{k}}. \]
For example, in the case of the Motzkin polynomial, it is clear that the monomials $x^6, y^4z^2$ and $y^2z^4$ are nonnegative for all $x,y,z \in \R$, and letting $x_1, x_2, x_3$ stand for these monomials respectively, the am-gm inequality  implies that
\[ x^6 + y^4z^2 + y^2z^4 \geq 3x^2y^2z^2 \mbox{ for all } x,y,z \in \R. \]
These polynomials are known to be extreme in the cone of nonnegative polynomials and they cannot be written as a sum of squares (sos)~\cite{Reznick}.

%In other words, these polynomials all lie outside the cone of SOS-representable polynomials.
%Furthemore, they are known to be extreme in the cone of nonnegative polynomials.

It would be interesting to study the separation problems associated with using such non-sos polynomials in column generation. We briefly present one separation algorithm for a family of polynomials whose nonnegativity is provable through the am-gm inequality and includes the Motzkin and Choi-Lam polynomials. This will be a relatively easy-to-solve integer program in itself, whose goal is to find a polynomial $q$ amongst this family which is to be added as our new ``nonnegative atom''.

The family of $n$-variate polynomials under consideration consists of polynomials with only $k+1$ nonzero coefficients, with $k$ of them equal to one, and one equal to $-k$. (Notice that the Motzkin and the Choi-Lam polynomials are of this form with $k$ equal to three and four respectively.) Let $m$ be the number of monomials in $p$. Given a dual vector $\mu$ of (\ref{eq:polycol}) of dimension $m$, one can check if there exists a nonnegative degree $2d$ polynomial $q(x)$ in our family such that $\mu \cdot \coef{q(x)} < 0.$ This can be done by solving the following integer program (we assume that $p(x) = \sum_{i=1}^m x^{\alpha_i}$):
\begin{eqnarray}
	\underset{c,y}{\min} && \sum_{i=1}^m  \mu_ic_i - \sum_{i=1}^m k\mu_iy_i \\\mbox{ s.t.} 
	&&\sum_{i:\alpha_i \mbox{ is even}} \alpha_i c_i = k \sum_{i=1}^m \alpha_iy_i, \nonumber \\
	&&\sum_{i=1}^m c_i = k, \nonumber\\
	&&\sum_{i=1}^m  y_i = 1, \nonumber\\
	&& c_i \in \{0,1\}, y_i \in \{0,1\}, i=1,\ldots, m, c_i=0 \mbox{ if } \alpha_i \mbox{ is not even.}\nonumber
\end{eqnarray}
Here, we have $\alpha_i\in\mathbb{N}^n$  and the variables $c_i, y_i$ form the coefficients of the polynomial $q(x) = \sum_{i=1}^m c_ix^{\alpha_i} - k\sum_{i=1}^m y_i x^{\alpha_i}$. The above integer program has $2m$ variables, but only $n+2$ constraints (not counting the integer constraints).     
If a polynomial $q(x)$ with a negative objective value is found, then one can add it as a new atom for column generation. In our specific randomly generated polynomial optimization examples,
such polynomials did not seem to help in our preliminary experiments. Nevertheless, it would be interesting to consider other instances and problem structures.

Similarly, in the column generation approach to obtaining inner approximations of the copositive cone, one need not stick to positive {\ghprime semidefinite} matrices. It is known that the $5\times 5$ ``Horn matrix''~\cite{burer2009difference} for example is extreme in the copositive cone but cannot be written as the sum of a nonnegative and a positive semidefinite matrix.
One could define a separation problem for a family of Horn-like matrices and add them in a column generation approach. Exploring such strategies is left for future research.

\section{Acknowledgments}
We are grateful to Anirudha Majumdar for insightful discussions and for his help with some of the numerical experiments in this paper.

\bibliographystyle{abbrv}
\bibliography{pablo_amirali,elib}

%\bibliographystyle{plain}
%\begin{thebibliography}{10}
%
%\bibitem{am}
%A. A. Ahmadi and A. Majumdar. 
%DSOS and SDSOS optimization: More tractable alternatives to SOS optimization.
%{\em Manuscript.}
%
%\bibitem{cl}
%  M. D. Choi and T. Y. Lam, Extremal positive semidefinite forms, {\em Math. Ann.} {\bf 231} 1--8 (1977).
%  
%\bibitem{dp}
%E. de Klerk and D. V. Pasechnik.  Approximation of the stability number of a graph via copositive programming. {\em SIAM Journal on Optimization}, {\bf 12}(4), 875-- 892 (2002).
%
%\bibitem{mot}
%T.S. Motzkin, The arithmetic-geometric inequality. Proc. Symposium on Inequalities (ed. O. Shisha), Academic Press, New York, 1967, pp. 205-–224.
%
%\bibitem{mk}
%K. G. Murty and S. N. Kabadi. Some NP-complete problems in quadratic and nonlinear
%programming. {\em Mathematical Programming}, {\bf 39} 117--129 (1987).
%  
%\bibitem{pp}
%P. A. Parrilo. Semidefinite programming relaxations for semialgebraic problems. Mathematical Programming {\bf 96} 293--320 (2003).
%\end{thebibliography}

\end{document}